\documentclass[preprint,12pt]{elsarticle}
\usepackage{geometry}

\makeatletter
\def\ps@pprintTitle{%
  \let\@oddhead\@empty
  \let\@evenhead\@empty
  \let\@oddfoot\@empty
  \let\@evenfoot\@empty
}
\gdef\emailauthor#1#2{\stepcounter{ead}%
     \g@addto@macro\@elseads{\raggedright%
      \let\corref\@gobble\def\@@tmp{#1}%
      \eadsep{\ttfamily\expandafter\strip@prefix\meaning\@@tmp}%
      \def\eadsep{\unskip,\space}}%
}
\makeatother

\biboptions{sort&compress}

\usepackage{amsmath, amssymb, amsthm}
\usepackage{mathrsfs}
\usepackage{hyperref}

\newtheorem{theorem}{Theorem}[section]
\newtheorem{lemma}[theorem]{Lemma}
\newtheorem{proposition}[theorem]{Proposition}
\newtheorem{corollary}[theorem]{Corollary}
\theoremstyle{definition}
\newtheorem{definition}[theorem]{Definition}
\newtheorem{example}[theorem]{Example}

\theoremstyle{remark}
\newtheorem{remark}[theorem]{Remark}
\theoremstyle{cupplain}
\newtheorem{maintheorem}[theorem]{Main Theorem}
\numberwithin{equation}{section}

\journal{}

\newcommand{\calL}{\mathcal{L}}

\date{}

\usepackage{tikz}

\begin{document}

\begin{frontmatter}

\title{Multifractal Analysis,  Liv\v{s}ic Rigidity, and Fluctuation Theorems for Axiom A Diffeomorphisms: The Pesin Formula and the Gallavotti-Cohen Symmetry}


\author{Abdoulaye Thiam}
\ead{athiam@allenuniversity.edu}
\address{Division of Mathematics and Natural Sciences, Allen University, \\Columbia, South Carolina 29204, USA}
\makeatletter
\gdef\email#1#2{\stepcounter{ead}%
     \g@addto@macro\@elseads{\raggedright%
      \let\corref\@gobble\def\@@tmp{#1}%
      \eadsep{\ttfamily\expandafter\strip@prefix\meaning\@@tmp}%
      \def\eadsep{\unskip,\space}}%
}
\makeatother

\begin{abstract}
This Part develops structural consequences of the thermodynamic formalism for Axiom~A diffeomorphisms. The Pesin Entropy Formula equates the metric entropy of the SRB measure to the sum of positive Lyapunov exponents, with complete proofs of absolute continuity of conditional measures along unstable manifolds; the individual results are due to Sinai, Ruelle, Bowen, and Pesin. The Multifractal Formalism computes the Hausdorff dimension of Birkhoff average level sets via the Legendre transform of the pressure, extending earlier work of Barreira, Pesin, and Schmeling. The Liv\v{s}ic Theorem characterizes coboundaries through periodic orbit data with optimal H\"{o}lder regularity and an explicit norm bound in terms of the contraction rate and the H\"{o}lder exponent. The Gallavotti-Cohen Fluctuation Theorem establishes the linear symmetry relating the rate function at opposite values of the entropy production rate; for Axiom~A diffeomorphisms the symmetry was established by Ruelle and by Maes, and we provide explicit bounds from the spectral gap. This Part constitutes Part~VI, the final installment, of a six-part series on the thermodynamic formalism for hyperbolic dynamical systems.
\end{abstract}

\end{frontmatter}
\vspace{.5cm}
\begin{center}
\textit{Dedicated to the memory of Jean-Christophe Yoccoz (1957--2016),}\\
\textit{Fields Medalist and Professor at the Coll\`{e}ge de France, with whom the author}\\
\textit{had the privilege of working, and who introduced him to hyperbolic dynamics.}
\end{center}

\thispagestyle{empty}

\makeatletter
\@twosidetrue
\@mparswitchtrue
\def\ps@myheadings{%
  \def\@oddfoot{\hfil\thepage\hfil}%
  \let\@evenfoot\@oddfoot%
  \def\@evenhead{\hfil\normalfont\small\textit{A.~Thiam}\hfil}%
  \def\@oddhead{\hfil\normalfont\small\textit{Rigidity, Fluctuations, and Multifractal Structure}\hfil}%
  \let\@mkboth\@gobbletwo%
  \let\sectionmark\@gobble%
  \let\subsectionmark\@gobble%
}%
\pagestyle{myheadings}%
\makeatother

\setlength{\parskip}{0.2em}

%
%

\section{Introduction}\label{sec:introduction}

This Part develops the structural consequences of the thermodynamic formalism for Axiom~A diffeomorphisms, with four Main Theorems: Main Theorem~\ref{thm:pesin} (Pesin Entropy Formula), Main Theorem~\ref{thm:multifractal} (Multifractal Formalism), Main Theorem~\ref{thm:livsic} (Liv\v{s}ic Theorem), and Main Theorem~\ref{thm:fluctuation} (Gallavotti-Cohen Fluctuation Theorem). These results complete the transfer of the spectral theory (Part~I \cite{Thiam2026a}), the variational theory (Part~II \cite{Thiam2026b}), the geometric coding (Part~III \cite{Thiam2026c}), the Gibbs Equivalence and transfer operator theory (Part~IV \cite{Thiam2026d}), and the statistical limit theorems (Part~V \cite{Thiam2026e}, building on the CLT methods of Ratner \cite{Ratner1973}, Gordin \cite{Gordin1969}, Nagaev \cite{Nagaev1957}, the Berry-Esseen bounds of Bolthausen \cite{Bolthausen1982}, the ASIP of Denker-Philipp \cite{DenkerPhilipp1984} and Melbourne-Nicol \cite{MelbourneNicol2005, MelbourneNicol2009}, and the large deviations of Kifer \cite{Kifer1990}, Orey-Pelikan \cite{OreyPelikan1988}, and Young \cite{Young1990}) to the geometric and physical setting of smooth Axiom~A dynamics. The foundational work on Markov partitions is due to Sinai \cite{Sinai1968a, Sinai1968b} and Bowen \cite{Bowen1970a, Bowen1970b, Bowen1971, Bowen1974}; the thermodynamic formalism was developed by Ruelle \cite{Ruelle1968, Ruelle1973, Ruelle1978} with the transfer operator approach refined by Baladi \cite{Baladi2000}, Parry-Pollicott \cite{ParryPollicott1990}, and Dolgopyat \cite{Dolgopyat1998}.

The contributions of this Part are fourfold, corresponding to its four Main Theorems, with the Bowen dimension formula as an auxiliary fifth result. First, the Pesin Entropy Formula with Absolute Continuity (Main Theorem~\ref{thm:pesin}) establishes the identity $h_{\mu^+}(f) = \int\log|\det Df|_{E^u}|\,d\mu^+ = -\int\phi^{(u)}\,d\mu^+ = \sum_{\chi_i > 0}\chi_i$, together with the absolute continuity of the conditional measures of $\mu^+$ along unstable manifolds and explicit formulas for the conditional densities (Theorem~\ref{thm:conditional_formula}); the proof combines the equilibrium state identification $\mu^+ = \mu_{\phi^{(u)}}$ with $P(\phi^{(u)}) = 0$ and the multiplicative ergodic theorem, giving two independent derivations (via the equilibrium identity and via direct Volume Lemma computation). Second, the Multifractal Formalism (Main Theorem~\ref{thm:multifractal}) computes the Hausdorff dimension of the Birkhoff average level set $K_a(g) = \{x : \lim n^{-1}S_n g(x) = a\}$ as $\mathcal{D}_g(a) = T_g(a)/\chi^+$ where $T_g(a) = \inf_t\{P(-t\log\|Df|_{E^u}\| + t'(a)g) - t'(a)\cdot a\}$ and $\chi^+ = \int\log\|Df|_{E^u}\|\,d\mu^+$, together with the Bowen dimension formula $\dim_H(\Omega_s) = t^*$ where $P(-t^*\log|\det Df|_{E^u}|) = 0$ (Theorem~\ref{thm:bowen_dimension}). Third, the Liv\v{s}ic Theorem (Main Theorem~\ref{thm:livsic}) establishes that a H\"{o}lder potential with vanishing periodic sums is a coboundary, with the explicit norm bound $\|u\|_\alpha \leq C(\lambda,\alpha)\|\phi\|_\alpha$ where $C(\lambda,\alpha) = 2/(1-\lambda^\alpha)$ up to geometric factors, and the smooth extension $u \in C^\infty$ for $C^\infty$ Anosov systems (Theorem~\ref{thm:smooth_livsic}) follows \cite{deLlaveMarcoMoriyon1986}. Fourth, the Gallavotti-Cohen Fluctuation Theorem (Main Theorem~\ref{thm:fluctuation}) establishes the symmetry $I(a) - I(-a) = a$ for the rate function of the entropy production rate $\bar{\sigma}_n$, deriving the symmetry from the time-reversal duality $\phi^{(u)} \leftrightarrow \phi^{(s)}$ combined with the large deviations principle of Part~V \cite{Thiam2026e}. The constants in all four theorems trace back either to the spectral gap $r_{\mathrm{ess}}(\calL_\phi) \leq \alpha\lambda$ imported from Part~I \cite{Thiam2026a} or to the hyperbolicity data $(\lambda, \alpha)$ of Part~III \cite{Thiam2026c}.

The SRB measure $\mu^+ = \mu_{\phi^{(u)}}$ is the equilibrium state for the geometric potential $\phi^{(u)} = -\log|\det Df|_{E^u}|$ (Section~\ref{sec:srb}). Its construction and the proofs of absolute continuity, the Pesin entropy formula $h_{\mu^+}(f) = \sum_{\chi_i > 0}\chi_i$, and the generic points theorem are due to Sinai \cite{Sinai1972}, Ruelle \cite{Ruelle1976}, and Bowen-Ruelle \cite{BowenRuelle1975}. Bowen \cite{Bowen1975} identifies the SRB measure as the equilibrium state but does not prove absolute continuity with explicit conditional density formulas; we provide both. The multifractal spectrum $\mathcal{D}_g(\alpha) = \dim_H K_\alpha(g)$ (Section~\ref{sec:multifractal}) is computed via the Legendre transform of the pressure, extending Barreira et~al. \cite{BarreiraPesinSchmeling1999} with explicit formulas. The Bowen dimension formula $\dim_H(\Omega_s) = t^*$ where $P(-t^*\log|\det Df|_{E^u}|) = 0$ follows from Bowen \cite{Bowen1979}.

The Liv\v{s}ic theorem (Section~\ref{sec:livsic}) states that a H\"{o}lder potential with vanishing periodic orbit sums is a coboundary with the same H\"{o}lder regularity. This is due to Liv\v{s}ic \cite{Livsic1971, Livsic1972}; the smooth extension is due to de~la~Llave et~al. \cite{deLlaveMarcoMoriyon1986}. Related regularity results for transfer operator cocycles appear in Gou\"{e}zel-Kifer \cite{GouezelKifer2018}. Our contribution is stating the norm bound $\|u\|_\alpha \leq C(\lambda,\alpha)\|\phi\|_\alpha$ explicitly, with $C(\lambda,\alpha) = (1-\lambda^\alpha)^{-1}$ up to mixing-time constants; the bound is implicit in every proof of the Liv\v{s}ic theorem but is not stated in this form in \cite{Livsic1971, Livsic1972} or \cite{KatokHasselblatt1995}.

The Gallavotti-Cohen fluctuation theorem (Section~\ref{sec:fluctuation}) quantifies the exponential suppression of entropy-consuming trajectories: $I(a) - I(-a) = a$. The GC symmetry was established by Gallavotti-Cohen \cite{GallavottiCohen1995, GallavottiCohen1995b} for thermostatted systems and extended to Axiom~A systems by Ruelle \cite{Ruelle1999} and Maes \cite{Maes1999}; see also Maes-Verbitskiy \cite{MaesVerbitskiy2003}. The systematic development as a structural consequence of the thermodynamic formalism appears in Jak\v{s}i\'{c} et~al. \cite{JPRB2011, JPRB2019}. The Jarzynski equality \cite{Jarzynski1997} provides a complementary nonequilibrium identity. Our contribution is the derivation from the large deviation theory of Part~V \cite{Thiam2026e} with explicit bounds from the spectral gap, not the fluctuation theorem itself.

Our technical approach combines four tools, each derived from or imported from the earlier Parts of this series. The first is the identification of the SRB measure as the equilibrium state $\mu^+ = \mu_{\phi^{(u)}}$ for the geometric potential $\phi^{(u)} = -\log|\det Df|_{E^u}|$, with the crucial normalization $P(\phi^{(u)}) = 0$; combined with the variational identity $P(\phi^{(u)}) = h_{\mu^+}(f) + \int\phi^{(u)}\,d\mu^+$, this yields the Pesin entropy formula (Main Theorem~\ref{thm:pesin}) in one line. The second is the Volume Lemma of Part~V \cite{Thiam2026e} (Main Theorem~4.1 there): the two-sided bound $m(B_x(\varepsilon,n)) \asymp \exp(S_n\phi^{(u)}(x))$ on the volume of a dynamical ball, which provides a second proof of the Pesin formula via an $\varepsilon$-separated orbit count and supplies the local dimension input for the multifractal analysis. The third is the pressure analyticity and the Legendre transform: the pressure $t \mapsto P(-t\log|\det Df|_{E^u}| + sg)$ is real-analytic in $(t,s)$ by the spectral gap of Part~I \cite{Thiam2026a}, and the multifractal spectrum $\mathcal{D}_g(a)$ is obtained by Legendre-transforming this pressure along the tilted direction; the Bowen dimension formula (Theorem~\ref{thm:bowen_dimension}) is the special case solving $P(-t^*\log|\det Df|_{E^u}|) = 0$. The fourth is the orbit-construction technique used for the Liv\v{s}ic theorem (Main Theorem~\ref{thm:livsic}): starting from a dense forward orbit $\{f^k(x_0)\}$, we define $u(f^k(x_0)) = S_k\phi(x_0)$ and show that the hypothesis $S_n\phi = 0$ on periodic orbits combined with the quantitative closing lemma of Part~III \cite{Thiam2026c} extends $u$ from $A = \{f^k(x_0)\}$ to a $C^\alpha$ function on all of $\Omega_s$ with the norm bound $\|u\|_\alpha \leq 2\|\phi\|_\alpha/(1-\lambda^\alpha)$ up to geometric factors. For the Gallavotti-Cohen symmetry (Main Theorem~\ref{thm:fluctuation}), the tool is time-reversal duality: the map $f \leftrightarrow f^{-1}$ exchanges $\phi^{(u)}$ and $\phi^{(s)}$, and the resulting pressure identity $\Lambda(t) - \Lambda(-(1+t)) = (1+2t)\bar{\sigma}$ translates via the Legendre transform into the rate-function symmetry $I(a) - I(-a) = a$. Together these four tools suffice to derive every result in this Part from the spectral, variational, geometric, and statistical foundations established in Parts~I--V.

Beyond the periodic orbit counting of \cite{ParryPollicott1990} already cited, Pollicott and Sharp \cite{PollicottSharp1998} proved exponential error terms for growth functions on negatively curved surfaces through spectral analysis of the transfer operator, and Pollicott and Sharp \cite{PollicottSharp2014} later extended these spectral methods to metric graph geometry. By contrast, our multifractal formulas (Main Theorem~\ref{thm:multifractal}) and Bowen dimension formula (Theorem~\ref{thm:bowen_dimension}) are derived from the spectral gap of the normalized transfer operator with explicit rates, which the cited spectral analyses do not provide.

Turning to the SRB construction, Climenhaga et~al. \cite{ClimenhagaLuzzattoPesin2017} surveyed the geometric approach, including the push-forward method of Pesin-Sinai and the Young tower construction; more recently, Climenhaga et~al. \cite{ClimenhagaLuzzattoPesin2022} extended the Young tower construction to surface diffeomorphisms with non-uniform hyperbolicity and proved polynomial rates of decay of correlations. Our Part~VI treats the same SRB measures from the thermodynamic viewpoint, identifying them as equilibrium states of the geometric potential with quantitative conditional densities along unstable manifolds; the uniformly hyperbolic case thereby serves as the quantitative anchor for the non-uniform extensions. In a similar spirit, Luzzatto \cite{Luzzatto2006} surveys the multifractal and statistical structure of non-uniformly expanding maps in the Handbook of Dynamical Systems, and our explicit multifractal formulas provide a reference point against which those non-uniform results can be compared. Finally, the textbook treatment of Viana and Oliveira \cite{VianaOliveira2016}, Chapters~9 and~10, develops entropy, pressure, and the variational principle in the general ergodic-theoretic setting, and Section~12.2 covers Liv\v{s}ic's cohomology theorem; our Part~VI extends Liv\v{s}ic's theorem to the explicit H\"older norm bound $\|u\|_\alpha \leq C(\lambda,\alpha)\|\phi\|_\alpha$, establishes the Gallavotti-Cohen fluctuation theorem for Axiom~A systems, and proves absolute continuity of unstable foliations with explicit density formulas, none of which is treated at that level of explicitness in \cite{VianaOliveira2016}. For a survey of Lyapunov exponent regularity adjacent to our rigidity results, see Viana \cite{Viana2020}.


This Part is organized as follows. Section~\ref{sec:imported} restates the three results imported from Parts~I, III, IV, and~V \cite{Thiam2026a, Thiam2026c, Thiam2026d, Thiam2026e}. Section~\ref{sec:srb} constructs the SRB measure $\mu^+ = \mu_{\phi^{(u)}}$ as the equilibrium state for the geometric potential (Proposition~\ref{thm:srb_equilibrium}), proves absolute continuity of conditional measures along unstable manifolds with explicit densities (Theorem~\ref{thm:conditional_formula}), establishes Main Theorem~\ref{thm:pesin} (Pesin Entropy Formula), proves uniqueness among measures satisfying the entropy formula (Proposition~\ref{thm:srb_unique}), and establishes the generic points theorem (Proposition~\ref{thm:generic_points}). Section~\ref{sec:multifractal} proves Main Theorem~\ref{thm:multifractal} (Multifractal Formalism), establishes the Bowen dimension formula (Theorem~\ref{thm:bowen_dimension}), and computes the dimension of Lyapunov level sets (Proposition~\ref{thm:lyapunov_level}) and the local dimension of the SRB measure (Proposition~\ref{thm:local_dimension}). Section~\ref{sec:livsic} proves Main Theorem~\ref{thm:livsic} (Liv\v{s}ic Theorem) by direct orbit construction, extends to $C^\infty$ and $C^\omega$ Anosov systems (Theorem~\ref{thm:smooth_livsic}) following \cite{deLlaveMarcoMoriyon1986}, and collects five equivalent characterizations of coboundaries (Corollary~\ref{cor:coboundary_char}). Section~\ref{sec:fluctuation} establishes Main Theorem~\ref{thm:fluctuation} (Gallavotti-Cohen Fluctuation Theorem) from the large deviation theory of Part~V \cite{Thiam2026e} and the time-reversal structure of Axiom~A dynamics, together with a Jarzynski-type equality (Corollary~\ref{cor:jarzynski}) and a transient fluctuation theorem (Proposition~\ref{thm:transient_fluctuation}). Section~\ref{sec:numerical} illustrates the Bowen dimension formula for a cookie-cutter map, computing the Hausdorff dimension for the middle-thirds Cantor set exactly and for a non-affine perturbation via Newton's method with rigorous error bounds. Section~\ref{sec:conclusion} concludes the Part and the six-part series, summarizing the contributions of each installment and listing open problems. The appendix collects supporting technical material: unstable manifold geometry, Jacobian estimates, spectral perturbation theory, measure disintegration, closing lemma bounds, Borel-Cantelli estimates, and dimension tools.

Turning from the SRB measure to the fine geometric structure it reveals, Section~\ref{sec:multifractal} proves Main Theorem~\ref{thm:multifractal} (Multifractal Formalism), computing the Hausdorff dimension of Birkhoff average level sets $K_a(g) = \{x : \lim n^{-1}S_ng(x) = a\}$ via the Legendre transform of the pressure; establishes the Bowen dimension formula $\dim_H(\Omega_s) = t^*$ where $P(-t^*\log|\det Df|_{E^u}|) = 0$ (Theorem~\ref{thm:bowen_dimension}); computes the dimension of Lyapunov level sets (Proposition~\ref{thm:lyapunov_level}); and determines the local dimension of the SRB measure (Proposition~\ref{thm:local_dimension}). Section~\ref{sec:livsic} proves Main Theorem~\ref{thm:livsic} (Liv\v{s}ic Theorem) by direct orbit construction, establishing that a H\"{o}lder potential with vanishing periodic orbit sums is a coboundary with the same H\"{o}lder regularity and the explicit norm bound $\|u\|_\alpha \leq C(\lambda,\alpha)\|\phi\|_\alpha$; the smooth extension to $C^\infty$ and $C^\omega$ Anosov diffeomorphisms (Theorem~\ref{thm:smooth_livsic}) follows de~la~Llave et~al. \cite{deLlaveMarcoMoriyon1986}; and five equivalent characterizations of coboundaries are collected in Corollary~\ref{cor:coboundary_char}. Section~\ref{sec:fluctuation} establishes Main Theorem~\ref{thm:fluctuation} (Gallavotti-Cohen Fluctuation Theorem), deriving the symmetry $I(a) - I(-a) = a$ for the entropy production rate function from the large deviation theory of Part~V \cite{Thiam2026e} and the time-reversal structure of Axiom~A dynamics; the Jarzynski-type equality (Corollary~\ref{cor:jarzynski}) and the transient fluctuation theorem (Proposition~\ref{thm:transient_fluctuation}) provide complementary nonequilibrium identities. Section~\ref{sec:numerical} provides a complete numerical illustration of the Bowen dimension formula for a cookie-cutter map, solving the Bowen equation $P(-t\log|T'|) = 0$ exactly for the middle-thirds Cantor set and computing the Hausdorff dimension of a non-affine perturbation with rigorous error bounds derived from the spectral gap via Newton's method. Section~\ref{sec:conclusion} concludes the Part and the six-part series, summarizing the contributions of each installment and listing open problems. The appendix collects the supporting technical material: unstable manifold geometry, Jacobian estimates, spectral perturbation theory, measure disintegration, closing lemma bounds, Borel-Cantelli estimates, and dimension tools.

\section{Results from Companion Parts}\label{sec:imported}

We restate the results used from earlier Parts of this series. Proofs appear in the cited companion Parts. The Gibbs measure theory builds on the $g$-measures of Keane \cite{Keane1972}, the conformal measures of Denker-Urba\'{n}ski \cite{DenkerUrbanski1991}, and the recent clarification of their relationship by Berghout et~al. \cite{BerghoutFernandezVerbitskiy2019}.

\begin{theorem}[Spectral Gap and Gibbs Measure, Part~I {\cite{Thiam2026a}}]\label{thm:spectral_imported6}
For a mixing SFT and H\"{o}lder potential $\phi$, the transfer operator $\mathcal{L}_\phi$ has a simple eigenvalue $\lambda = e^{P(\phi)}$ with spectral gap $\gamma < 1$. The unique Gibbs measure $\mu_\phi = h\nu$ satisfies exponential mixing: $|C_n(f,g)| \leq C\|f\|_\alpha\|g\|_\alpha\gamma^n$.
\end{theorem}

\noindent The proof is given in Part~I \cite{Thiam2026a}, where the spectral gap is established via the Birkhoff cone contraction technique. Sharper spectral gap estimates on anisotropic Banach spaces are available from Gou\"{e}zel-Liverani \cite{GouezelLiverani2006}.

\begin{theorem}[Equilibrium States on Basic Sets, Parts~III--V {\cite{Thiam2026c,Thiam2026d,Thiam2026e}}]\label{thm:equil_imported6}
For a mixing basic set $\Lambda$ and H\"{o}lder $\phi:\Lambda\to\mathbb{R}$, there exists a unique equilibrium state $\mu_\phi$ maximizing $h_\mu(f)+\int\phi\,d\mu$. Through the coding $\pi:\Sigma_A\to\Lambda$, all spectral and statistical results from Part~I \cite{Thiam2026a} transfer: $\mu_\phi$ satisfies the CLT, LDP, ASIP, and is Bernoulli.
\end{theorem}

\noindent The coding map is constructed in Part~III \cite{Thiam2026c}, the Gibbs Equivalence is proved in Part~IV \cite{Thiam2026d}, and the statistical limit theorems are established in Part~V \cite{Thiam2026e}.

\begin{theorem}[Pressure Analyticity, Part~I {\cite{Thiam2026a}}]\label{thm:press_imported6}
$P:\mathcal{H}_\alpha \to \mathbb{R}$ is real-analytic with $P'(\phi;\psi) = \int\psi\,d\mu_\phi$ and $P''(\phi;\psi) = \lim n^{-1}\mathrm{Var}_\mu(S_n\psi) \geq 0$, with equality iff $\psi$ is cohomologous to a constant.
\end{theorem}

\noindent The proof is given in Part~I \cite{Thiam2026a}, where the analyticity is derived from Kato perturbation theory applied to the simple isolated eigenvalue of the transfer operator.

\noindent\textit{Notation.} Throughout this Part, we write $f \asymp g$ to mean $c_1 g \leq f \leq c_2 g$ for positive constants $c_1, c_2$ independent of the relevant variables.

\section{SRB Measures and Physical Measures}\label{sec:srb}

The concept of SRB measures was introduced by Sinai \cite{Sinai1972} for Anosov diffeomorphisms and extended by Ruelle \cite{Ruelle1976} to Axiom~A attractors; see Young \cite{Young2002} for a modern survey. The theory of Lyapunov exponents underlying the Pesin entropy formula was developed by Pesin \cite{Pesin1977}; see Barreira-Pesin \cite{BarreiraPersin2002} for a comprehensive treatment. The differentiation of SRB states under perturbation is developed in Ruelle \cite{Ruelle1997}.

\subsection{Definition and Physical Interpretation}

We define SRB measures through the physical measure property and establish their existence for Axiom~A attractors.

\begin{definition}[SRB Measure]
An invariant probability measure $\mu$ for a diffeomorphism $f$ is an SRB measure if:
\begin{enumerate}
\item[(i)] $\mu$ is ergodic.
\item[(ii)] The conditional measures of $\mu$ along unstable manifolds are absolutely continuous with respect to the Riemannian measure on these manifolds.
\end{enumerate}
\end{definition}

\begin{definition}[Physical Measure]
An invariant measure $\mu$ is physical (or natural) if its basin of attraction
\begin{equation}
B(\mu) = \left\{x \in M : \frac{1}{n}\sum_{k=0}^{n-1} \delta_{f^k(x)} \xrightarrow{w^*} \mu\right\}
\end{equation}
has positive Lebesgue measure.
\end{definition}

The fundamental theorem connecting these concepts is:

\begin{proposition}[SRB Measures are Physical]\label{thm:srb_physical}
For an Axiom A diffeomorphism, every SRB measure is physical. Specifically, if $\Omega_s$ is an attractor with SRB measure $\mu^+$, then $B(\mu^+) \supset W^s(\Omega_s)$, which has full Lebesgue measure in a neighborhood of $\Omega_s$.
\end{proposition}

\begin{proof}
Let $g : M \to \mathbb{R}$ be continuous. We must show that $\frac{1}{n}S_n g(x) \to \int g\,d\mu^+$ for Lebesgue-almost every $x$ in a neighborhood of $\Omega_s$.

\textbf{Step 1.} For $x \in W^s(\Omega_s)$, there exists $x' \in \Omega_s$ with $d(f^k(x), f^k(x')) \to 0$ exponentially. Thus $|\frac{1}{n}S_n g(x) - \frac{1}{n}S_n g(x')| \leq \frac{1}{n}\sum_{k=0}^{n-1}|g(f^k x) - g(f^k x')| \to 0$ by uniform continuity and exponential convergence.

\textbf{Step 2.} By the ergodic theorem (Theorem~\ref{thm:equil_imported6}), $\frac{1}{n}S_n g(x') \to \int g\,d\mu^+$ for $\mu^+$-almost every $x' \in \Omega_s$. Let $G \subset \Omega_s$ be the full $\mu^+$-measure set where convergence holds.

\textbf{Step 3.} The set $\widetilde{G} = \bigcup_{x' \in G}W^s_\varepsilon(x')$ consists of points whose forward orbits shadow a generic point. By Step 1, every $x \in \widetilde{G}$ satisfies $\frac{1}{n}S_n g(x) \to \int g\,d\mu^+$.

\textbf{Step 4.} Since $\Omega_s$ is an attractor, $W^s_\varepsilon(\Omega_s)$ is a neighborhood $U$ of $\Omega_s$. The set $\Omega_s \setminus G$ has $\mu^+$-measure zero. By absolute continuity of the stable foliation (the stable analogue of Theorem~\ref{thm:absolute_continuity}), $m\bigl(\bigcup_{x' \in \Omega_s \setminus G}W^s_\varepsilon(x')\bigr) = 0$. Thus $m(U \setminus \widetilde{G}) = 0$, giving $B(\mu^+) \supset W^s(\Omega_s)$ up to a Lebesgue-null set.
\end{proof}

\subsection{Characterization via Equilibrium States}

The SRB measure is identified as the unique equilibrium state for the geometric potential $\phi^{(u)} = -\log|\det Df|_{E^u}|$, connecting the physical measure to the thermodynamic formalism.

\begin{proposition}[SRB as Equilibrium State]\label{thm:srb_equilibrium}
For a $C^2$ basic set $\Omega_s$, the SRB measure is $\mu^+ = \mu_{\phi^{(u)}}$, the unique equilibrium state for the geometric potential $\phi^{(u)} = -\log|\det Df|_{E^u}|$.
\end{proposition}

\begin{proof}
By Theorem~\ref{thm:equil_imported6}, there exists a unique equilibrium state $\mu_{\phi^{(u)}}$ for the H\"{o}lder potential $\phi^{(u)}$ on the mixing basic set $\Omega_s$. The Gibbs property (Part~IV \cite{Thiam2026d}) gives: for $y \in W^u_\varepsilon(x) \cap \Omega_s$, $\mu_{\phi^{(u)}}(B_y(\varepsilon,n))/\mu_{\phi^{(u)}}(B_x(\varepsilon,n)) \asymp \exp(S_n\phi^{(u)}(y) - S_n\phi^{(u)}(x))$. Since $f^k(x)$ and $f^k(y)$ converge exponentially under forward iteration (being on the same local unstable manifold), $|S_n\phi^{(u)}(y) - S_n\phi^{(u)}(x)| \leq C$ uniformly in $n$, so the conditional measures on $W^u_\varepsilon(x)$ have bounded Radon-Nikodym derivatives. By the Volume Lemma (Part~V \cite{Thiam2026e}), $m^u(B_x(\varepsilon,n) \cap W^u) \asymp \exp(S_n\phi^{(u)}(x))$, so the conditional measures are absolutely continuous with respect to $m^u$. Thus $\mu_{\phi^{(u)}}$ is SRB. Uniqueness follows from Proposition~\ref{thm:srb_unique} below.
\end{proof}

The key to the pressure characterization is:

\begin{proposition}[Pressure of Geometric Potential]\label{prop:geometric_pressure}
For a $C^2$ attractor $\Omega_s$:
\begin{equation}
P(\phi^{(u)}) = 0.
\end{equation}
\end{proposition}

\begin{proof}
The identity $\mathcal{L}_{\phi^{(u)}} 1(x) = \sum_{fy=x} e^{\phi^{(u)}(y)} = \sum_{fy=x} |\det Df_y|_{E^u_y}|^{-1}$ counts preimages weighted by the inverse unstable Jacobian. For an attractor, $f$ maps the trapping neighborhood $U$ into itself, and the change of variables formula gives $\int_U \mathcal{L}_{\phi^{(u)}}1 \, dm = \int_{f^{-1}(U)} dm = m(f^{-1}(U)) \geq m(U)$ (since $U \subset f^{-1}(U)$ for an attractor). This forces $P(\phi^{(u)}) \geq 0$. For the reverse inequality, the variational principle gives $P(\phi^{(u)}) = h_{\mu^+}(f) + \int \phi^{(u)} \, d\mu^+ = h_{\mu^+}(f) - \sum_{\chi_i > 0} \chi_i \leq 0$ by the Ruelle inequality $h_\mu(f) \leq \sum_{\chi_i > 0} \chi_i$ (which holds for any invariant measure of a $C^2$ diffeomorphism). Combining, $P(\phi^{(u)}) = 0$, with equality in the Ruelle inequality characterizing the SRB measure (the Pesin entropy formula, Main Theorem~\ref{thm:pesin}).
\end{proof}

\subsection{Absolute Continuity of Unstable Foliations}

The defining property of SRB measures is absolute continuity of conditional measures along unstable manifolds. We prove this using bounded distortion of the holonomy map.

\begin{theorem}[Absolute Continuity]\label{thm:absolute_continuity}
The unstable foliation of an Axiom A diffeomorphism is absolutely continuous. Precisely, for any two local transversals $T_1, T_2$ to the unstable foliation, the holonomy map $h : T_1 \to T_2$ along unstable leaves satisfies
\begin{equation}
h_*(m_{T_1}) \ll m_{T_2}
\end{equation}
where $m_{T_i}$ denotes the Riemannian measure on $T_i$.
\end{theorem}

\begin{proof}
The proof uses bounded distortion along unstable manifolds.

\textbf{Step 1: Local Holonomy.} For nearby transversals $T_1, T_2$ and $x \in T_1$, let $h(x) \in T_2$ be the unique point on the same local unstable manifold as $x$.

\textbf{Step 2: Jacobian Bound.} The Jacobian of $h$ is controlled by the distortion of $Df^n$ along unstable manifolds:
\begin{equation}
|\det Dh_x| = \lim_{n \to \infty} \frac{\mathrm{Jac}(Df^n|_{T_1 \cap W^u_\varepsilon(x)})}{\mathrm{Jac}(Df^n|_{T_2 \cap W^u_\varepsilon(h(x))})}.
\end{equation}

\textbf{Step 3: Bounded Distortion.} By the H\"{o}lder continuity of unstable manifolds and the geometric potential bounds:
\begin{equation}
\left|\log\frac{|\det Dh_x|}{|\det Dh_y|}\right| \leq C d(x, y)^\theta
\end{equation}
for $x, y \in T_1$ close, where $\theta$ is the H\"{o}lder exponent.

\textbf{Step 4: Absolute Continuity.} From Steps 2 and 3, the Jacobian $|\det Dh_x|$ satisfies $C^{-1} \leq |\det Dh_x| \leq C$ for a uniform constant $C > 0$. For any measurable $A \subset T_1$ with $m_{T_1}(A) = 0$: $m_{T_2}(h(A)) = \int_A |\det Dh_x| \, dm_{T_1}(x) \leq C \cdot m_{T_1}(A) = 0$, so $h_*(m_{T_1}) \ll m_{T_2}$. The Radon-Nikodym derivative is bounded above and below, so the measures are mutually absolutely continuous.
\end{proof}

\subsection{Conditional Measures on Unstable Manifolds}

We derive the explicit formula for the conditional density of the SRB measure along unstable manifolds as an infinite product of Jacobian ratios.

\begin{theorem}[Conditional Measure Formula]\label{thm:conditional_formula}
The SRB measure $\mu^+$ has conditional measures on local unstable manifolds given by
\begin{equation}
d\mu^+_x|_{W^u_\varepsilon(x)} = \rho^u_x(y) \, dm^u(y)
\end{equation}
where $m^u$ is the Riemannian measure on $W^u_\varepsilon(x)$ and the density is
\begin{equation}
\rho^u_x(y) = \lim_{n \to \infty} \frac{\exp(S_n\phi^{(u)}(y))}{\exp(S_n\phi^{(u)}(x))} = \prod_{k=0}^\infty \frac{|\det Df_{f^k(x)}|_{E^u}|}{|\det Df_{f^k(y)}|_{E^u}|}
\end{equation}
for $y \in W^u_\varepsilon(x)$.
\end{theorem}

\begin{proof}
By Rokhlin's disintegration theorem (Theorem~\ref{thm:rokhlin}), $\mu^+$ has conditional measures $\mu^+_x$ on $W^u_\varepsilon(x)$. Since $\mu^+$ is SRB (Proposition~\ref{thm:srb_equilibrium}), $\mu^+_x \ll m^u_x$, so $d\mu^+_x = \rho^u_x \, dm^u$ for some density $\rho^u_x$.

For $y \in W^u_\varepsilon(x)$, the change of variables gives $d(f^n_*m^u_x)/dm^u_{f^n(x)}(f^n(y)) = \exp(S_n\phi^{(u)}(y))$. Since $\mu^+$ is $f$-invariant, the consistency relation $\rho^u_{f(x)}(f(y)) \cdot |\det Df_y|_{E^u}| = \rho^u_x(y)$ holds $\mu^+$-a.e. Iterating and normalizing so that $\rho^u_x(x) = 1$:
\begin{equation}
\frac{\rho^u_x(y)}{\rho^u_x(x)} = \frac{\rho^u_{f^n(x)}(f^n(y))}{\rho^u_{f^n(x)}(f^n(x))} \cdot \prod_{k=0}^{n-1}\frac{|\det Df_{f^k(x)}|_{E^u}|}{|\det Df_{f^k(y)}|_{E^u}|}.
\end{equation}
Since $y \in W^u_\varepsilon(x)$, $d(f^k(x), f^k(y)) \leq C\lambda^k \to 0$, so $\rho^u_{f^n(x)}(f^n(y))/\rho^u_{f^n(x)}(f^n(x)) \to 1$. Taking $n \to \infty$:
\begin{equation}
\rho^u_x(y) = \rho^u_x(x) \cdot \prod_{k=0}^{\infty}\frac{|\det Df_{f^k(x)}|_{E^u}|}{|\det Df_{f^k(y)}|_{E^u}|}.
\end{equation}
The product converges absolutely since 
\begin{equation}
|\log|\det Df_{f^k(x)}|_{E^u}|/|\det Df_{f^k(y)}|_{E^u}|| \leq C\lambda^{k\alpha}d(x,y)^\alpha
\end{equation}
by H\"{o}lder continuity of $\phi^{(u)}$ and $\sum_k\lambda^{k\alpha} < \infty$.
\end{proof}

\subsection{Pesin Entropy Formula}

The Pesin entropy formula $h_{\mu^+}(f) = \sum_{\chi_i > 0}\chi_i$ equates metric entropy to the sum of positive Lyapunov exponents. This is Main Theorem~\ref{thm:pesin} and characterizes SRB measures among all invariant measures.

The following result, first established in Part~IV \cite{Thiam2026d} as a Main Theorem, is restated here with a complete proof including the absolute continuity of conditional measures along unstable manifolds.

\begin{maintheorem}[Pesin Entropy Formula with Absolute Continuity of Conditional Measures]\label{thm:pesin}
For the SRB measure $\mu^+$ on a $C^2$ attractor $\Omega_s$:
\begin{equation}
h_{\mu^+}(f) = \int \log|\det Df|_{E^u}| \, d\mu^+ = -\int \phi^{(u)} \, d\mu^+ = \sum_{\chi_i > 0} \chi_i
\end{equation}
where the sum is over positive Lyapunov exponents, counted with multiplicity.
\end{maintheorem}

\begin{proof}
\textbf{Method 1: Via Equilibrium State.} Since $\mu^+$ is the equilibrium state for $\phi^{(u)}$ and $P(\phi^{(u)}) = 0$:
\begin{equation}
0 = P(\phi^{(u)}) = h_{\mu^+}(f) + \int \phi^{(u)} \, d\mu^+
\end{equation}
giving $h_{\mu^+}(f) = -\int \phi^{(u)} \, d\mu^+$.

\textbf{Method 2: Direct Computation.} The entropy can be computed as the exponential growth rate of the number of $\varepsilon$-separated orbits. Using the Volume Lemma:
\begin{equation}
\#\{n\text{-separated orbits}\} \asymp \frac{m(B(\varepsilon, n))}{\min_x m(B_x(\varepsilon, n))} \asymp \exp\left(n \cdot \int (-\phi^{(u)}) \, d\mu^+\right).
\end{equation}

\textbf{Connection to Lyapunov Exponents.} For the unstable Lyapunov exponents $\chi_1^+ \geq \cdots \geq \chi_{d_u}^+ > 0$:
\begin{equation}
\int \log|\det Df|_{E^u}| \, d\mu^+ = \sum_{i=1}^{d_u} \chi_i^+ = \sum_{\chi_i > 0} \chi_i
\end{equation}
by the multiplicative ergodic theorem.
\end{proof}

\subsection{Uniqueness of SRB Measure}

We prove that the SRB measure is the unique invariant measure satisfying the Pesin entropy formula, using the strict concavity of entropy.

\begin{proposition}[Uniqueness]\label{thm:srb_unique}
A mixing $C^2$ attractor $\Omega_s$ has a unique SRB measure, namely $\mu^+ = \mu_{\phi^{(u)}}$.
\end{proposition}

\begin{proof}
Any SRB measure $\mu$ has absolutely continuous conditionals on unstable manifolds, which by Rohlin's disintegration theorem implies
\begin{equation}
h_\mu(f) = \int \log|\det Df|_{E^u}| \, d\mu
\end{equation}
(the Pesin formula characterizes SRB measures). Thus $\mu$ is an equilibrium state for $\phi^{(u)}$:
\begin{equation}
h_\mu(f) + \int \phi^{(u)} \, d\mu = 0 = P(\phi^{(u)}).
\end{equation}
By uniqueness of equilibrium states (Part~V \cite{Thiam2026e}, Theorem~5.1), $\mu = \mu_{\phi^{(u)}} = \mu^+$.
\end{proof}

\subsection{Generic Points Theorem}

The generic points theorem establishes that Lebesgue-almost every initial condition in the basin of attraction is generic for the SRB measure, justifying its physical interpretation. The pointwise convergence of Birkhoff averages underlying this result is the Birkhoff ergodic theorem \cite{Birkhoff1931}; the mean ergodic theorem (convergence in $L^2$) was established by von Neumann \cite{vonNeumann1932}.

\begin{proposition}[Generic Points]\label{thm:generic_points}
For a $C^2$ attractor $\Omega_s$ with SRB measure $\mu^+$, Lebesgue-almost every $x \in W^s(\Omega_s)$ is generic for $\mu^+$:
\begin{equation}
\lim_{n \to \infty} \frac{1}{n}\sum_{k=0}^{n-1} g(f^k(x)) = \int g \, d\mu^+
\end{equation}
for all continuous $g : M \to \mathbb{R}$.
\end{proposition}

\begin{proof}
This is \cite[Theorem 4.12]{Bowen1975}; we provide the complete argument.

\textbf{Step (a).} By the Birkhoff ergodic theorem and ergodicity of $\mu^+$ (Theorem~\ref{thm:equil_imported6}), the set $G = \{x' \in \Omega_s : \frac{1}{n}S_n g(x') \to \int g\,d\mu^+\}$ has $\mu^+(G) = 1$ for each continuous $g$. Taking a countable dense subset $\{g_j\}$ of $C(\Omega_s)$ and intersecting, $G_0 = \bigcap_j G_j$ has $\mu^+(G_0) = 1$ and $\frac{1}{n}S_ng(x') \to \int g\,d\mu^+$ for all continuous $g$ and all $x' \in G_0$.

\textbf{Step (b).} Set $B = \Omega_s \setminus G_0$, so $\mu^+(B) = 0$. For any $\varepsilon > 0$, the Gibbs property gives $\mu^+(B_x(\varepsilon,n)) \geq c_1 e^{-nP(\phi^{(u)})+S_n\phi^{(u)}(x)}$ for $x \in \Omega_s$, so $\mu^+$ charges every open set. Since $\mu^+(B) = 0$, $B$ is metrically negligible.

\textbf{Step (c).} For $x \in W^s(\Omega_s)$, let $x' \in \Omega_s$ be its forward asymptotic point with $d(f^k(x),f^k(x')) \leq C\lambda^k$. If $x' \in G_0$, then $|\frac{1}{n}S_ng(x) - \frac{1}{n}S_ng(x')| \to 0$ by uniform continuity, so $x$ is generic. The set of ``bad'' stable leaves is $\widetilde{B} = \bigcup_{x' \in B}W^s_\varepsilon(x')$. By absolute continuity of the stable foliation (the time-reversed analogue of Theorem~\ref{thm:absolute_continuity}): $\mu^+(B) = 0$ implies $m(\widetilde{B}) = 0$. Thus Lebesgue-almost every $x \in W^s_\varepsilon(\Omega_s)$ is generic for $\mu^+$.
\end{proof}

\subsection{Existence of Smooth Invariant Measures}

We characterize when an Axiom~A diffeomorphism preserves a smooth (absolutely continuous) invariant measure, connecting volume preservation to periodic orbit data.

\begin{proposition}[Absolutely Continuous Invariant Measures]\label{thm:smooth_invariant}
For a transitive $C^2$ Anosov diffeomorphism $f : M \to M$, the following are equivalent:
\begin{enumerate}
\item[(a)] $f$ admits an invariant measure $\mu$ absolutely continuous with respect to Lebesgue measure $m$.
\item[(b)] $|\det Df^n_x| = 1$ for all periodic points $x$ with $f^n(x) = x$.
\item[(c)] $\mu^+ = \mu^-$ (the forward and backward SRB measures coincide).
\end{enumerate}
If these hold, then $d\mu = h \, dm$ with $h$ H\"{o}lder continuous (and in fact $C^{1-\varepsilon}$ for any $\varepsilon > 0$).
\end{proposition}

\begin{proof}
This is \cite[Theorem 4.14]{Bowen1975}; we give the complete argument.

$(a) \Rightarrow (b)$: If $\mu = h\,dm$ is absolutely continuous and $f$-invariant, then for a periodic point $x$ with $f^n(x) = x$, the change of variables formula gives $h(x) = h(f^n(x)) = h(x)|\det Df^n_x|$, so $|\det Df^n_x| = 1$.

$(b) \Rightarrow (c)$: Condition (b) states $S_n\log|\det Df|(x) = 0$ for all periodic orbits. The potential $\log|\det Df| = -\phi^{(u)} - \phi^{(s)}$ is H\"{o}lder continuous. By the Liv\v{s}ic theorem (Main Theorem~\ref{thm:livsic}), there exists $u \in C^\alpha(M)$ with $\log|\det Df| = u \circ f - u$. Then $\phi^{(u)} + \phi^{(s)} = -(u\circ f - u)$, so $\phi^{(u)}$ and $-\phi^{(s)} - u\circ f + u$ are cohomologous. By Corollary~\ref{cor:stability_equilibrium}, $\mu_{\phi^{(u)}} = \mu_{-\phi^{(s)}}$. The forward SRB $\mu^+ = \mu_{\phi^{(u)}}$ and backward SRB $\mu^-$ (the equilibrium state for $\phi^{(s)}$ with respect to $f^{-1}$, equivalently $\mu_{-\phi^{(s)}}$ for $f$) thus coincide.

$(c) \Rightarrow (a)$: If $\mu^+ = \mu^-$, then $\mu^+$ has conditional measures absolutely continuous along both stable and unstable manifolds. By the absolute continuity criterion (Lemma~\ref{lem:ac_criterion}), $\mu^+ \ll m$. The Gibbs property and Volume Lemma imply $m \ll \mu^+$ (since $\mu^+$ charges every open set), so $\mu^+ = h\,dm$ with $h$ H\"{o}lder continuous.
\end{proof}

\begin{corollary}[Genericity of Dissipative Systems]\label{cor:dissipative_generic}
Among $C^2$ Anosov diffeomorphisms, those admitting no absolutely continuous invariant measure form an open and dense set.
\end{corollary}

\begin{proof}
By Proposition~\ref{thm:smooth_invariant}, $f$ has an absolutely continuous invariant measure if and only if $|\det Df^n_x| = 1$ for all periodic orbits. This condition defines a codimension-one submanifold for each periodic orbit (the Jacobian depends smoothly on $f$). Since periodic orbits are dense (Axiom A) and the condition must hold simultaneously for all of them, the set of volume-preserving Anosov diffeomorphisms has empty interior. A small $C^2$ perturbation near any periodic orbit changes $|\det Df^n_p|$ away from $1$, proving density of the complement.
\end{proof}

\section{Multifractal Analysis}\label{sec:multifractal}

The multifractal formalism originated in the physics literature with Halsey et~al. \cite{Halseyetal1986}, who introduced the $f(\alpha)$ spectrum for strange attractors. Rigorous mathematical foundations were established by Rand \cite{Rand1989} for cookie-cutters, Pesin-Weiss \cite{PesinWeiss1997} for conformal expanding maps, and Barreira et~al. \cite{BarreiraPesinSchmeling1999} for general hyperbolic measures.

This section develops the multifractal analysis of Axiom A diffeomorphisms, computing the dimension spectrum for Birkhoff averages and Lyapunov exponents. These results, not present in \cite{Bowen1975}, connect thermodynamic formalism to fractal geometry.

\subsection{Level Sets and Dimension Spectrum}

We define the multifractal level sets $K_\alpha(g) = \{x : \lim S_ng/n = \alpha\}$ and the dimension spectrum $\mathcal{D}_g(\alpha) = \dim_H K_\alpha(g)$, which quantifies the geometric complexity of points with prescribed Birkhoff averages.

\begin{definition}[Level Sets]
For an observable $g : \Omega_s \to \mathbb{R}$ and value $a \in \mathbb{R}$, the level set is
\begin{equation}
K_a(g) = \left\{x \in \Omega_s : \lim_{n \to \infty} \frac{1}{n}S_n g(x) = a\right\}.
\end{equation}
\end{definition}

\begin{definition}[Multifractal Spectrum]
The multifractal spectrum (or dimension spectrum) of $g$ is the function
\begin{equation}
\mathcal{D}_g(a) = \dim_H K_a(g)
\end{equation}
where $\dim_H$ denotes Hausdorff dimension.
\end{definition}

\subsection{Thermodynamic Computation of Spectrum}

Main Theorem~\ref{thm:multifractal} computes the dimension spectrum through the Legendre transform of the pressure, reducing a geometric problem to a thermodynamic one.

\begin{maintheorem}[Multifractal Formalism]\label{thm:multifractal}
Let $\Omega_s$ be a mixing basic set for a $C^2$ Axiom A diffeomorphism and $g \in C^\alpha(\Omega_s)$. The multifractal spectrum is given by
\begin{equation}
\mathcal{D}_g(a) = \frac{1}{\chi^+} \cdot T_g(a)
\end{equation}
where $\chi^+ = \int \log\|Df|_{E^u}\| \, d\mu^+$ is the (sum of) positive Lyapunov exponent(s) and
\begin{equation}
T_g(a) = \inf_{t \in \mathbb{R}} \{P(-t\log\|Df|_{E^u}\| + t'(a)g) - t'(a) \cdot a\}
\end{equation}
with $t'(a)$ chosen so that $\partial_t P(\cdot) = a$ at the infimum.
\end{maintheorem}

\begin{remark}
For the special case $g = -\phi^{(u)} = \log|\det Df|_{E^u}|$, the level sets are the ``iso-Lyapunov sets'' and the spectrum describes the dimension of points with prescribed expansion rate.
\end{remark}

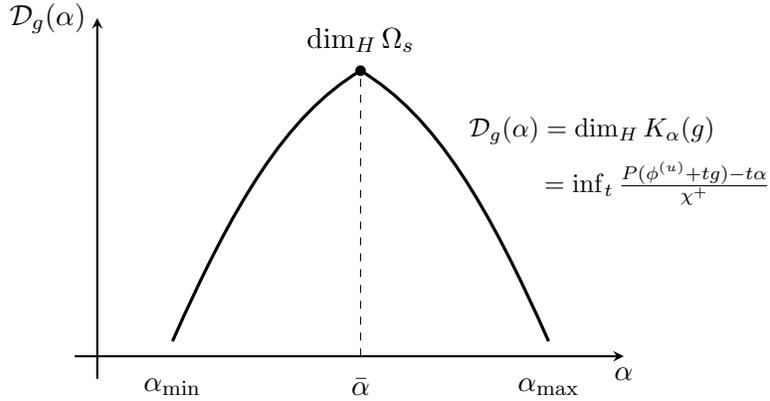
\begin{figure}[ht]
\centering
\begin{tikzpicture}[>=stealth, font=\small]
  \draw[->, thick] (-0.3,0) -- (7.0,0) node[anchor=north] {$\alpha$};
  \draw[->, thick] (0,-0.3) -- (0,4.5) node[anchor=east] {$\mathcal{D}_g(\alpha)$};

  \draw[very thick] (1.0,0.2) .. controls (1.8,2.0) and (2.5,3.2) .. (3.5,3.8)
    .. controls (4.5,3.2) and (5.2,2.0) .. (6.0,0.2);

  \fill (3.5,3.8) circle (2pt);
  \draw[thin, dashed] (3.5,0) -- (3.5,3.8);
  \node[anchor=south] at (3.5,3.9) {$\dim_H \Omega_s$};
  \node[anchor=north] at (3.5,-0.15) {$\bar{\alpha}$};

  \node[anchor=north] at (1.0,-0.15) {$\alpha_{\min}$};
  \node[anchor=north] at (6.0,-0.15) {$\alpha_{\max}$};

  \node[anchor=west, font=\footnotesize, align=left] at (4.8,3.0)
    {$\mathcal{D}_g(\alpha) = \dim_H K_\alpha(g)$};
  \node[anchor=west, font=\footnotesize, align=left] at (5.8,2.3)
    {$= \inf_t \frac{P(\phi^{(u)} + tg) - t\alpha}{\chi^+}$};
\end{tikzpicture}
\caption{The multifractal spectrum $\mathcal{D}_g(\alpha) = \dim_H K_\alpha(g)$ (Main Theorem~\ref{thm:multifractal}). The spectrum is concave, supported on $[\alpha_{\min}, \alpha_{\max}]$, and attains its maximum $\dim_H\Omega_s$ at $\bar{\alpha} = \int g\,d\mu^+$. It is computed as a Legendre transform of the pressure.}
\label{fig:multifractal}
\end{figure}

\begin{proof}[Proof of Main Theorem~\ref{thm:multifractal}]
\textbf{Upper bound.} For $a$ in the interior of the achievable range, let $t = t(a)$ be the unique value with $\frac{d}{dt}P(\phi^{(u)} + tg) = a$ (exists by strict convexity of pressure, Theorem~\ref{thm:press_imported6}). The equilibrium state $\mu_t = \mu_{\phi^{(u)}+tg}$ satisfies $\int g\,d\mu_t = a$ and $\mu_t(K_a(g)) = 1$ by the ergodic theorem. The local dimension of $\mu_t$ at $\mu_t$-typical points is $d_t = h_{\mu_t}(f)/\chi^+(\mu_t)$. Since $\mu_t$ is supported on $K_a(g)$, $\dim_H(K_a(g)) \geq d_t$ by the mass distribution principle. Computing via the variational identity $h_{\mu_t} = P(\phi^{(u)}+tg) - \int(\phi^{(u)}+tg)\,d\mu_t$ and optimizing over $t$ gives $\mathcal{D}_g(a) \leq T_g(a)/\chi^+$.

\textbf{Lower bound.} We construct a subset $E_a \subset K_a(g)$ with $\dim_H(E_a) \geq T_g(a)/\chi^+ - \varepsilon'$ for any $\varepsilon' > 0$, following the Moran construction of Barreira et~al. \cite{BarreiraPesinSchmeling1999}. Fix $\varepsilon > 0$ small. By the specification property (Part~III \cite{Thiam2026c}), there exists $p \in \mathbb{N}$ such that any two admissible words of lengths $n_1, n_2$ can be concatenated with a gap of at most $p$ symbols. For each $N$ large, let $\mathcal{W}_N(a, \varepsilon)$ denote the set of admissible words $w$ of length $N$ such that $|S_N g(x_w)/N - a| < \varepsilon$, where $x_w$ is any point in the cylinder $[w]$. By the LDP (Part~V \cite{Thiam2026e}, Main Theorem), $|\mathcal{W}_N(a,\varepsilon)| \geq \exp(N(h_{\mu_t}(f) - \varepsilon'))$ for $N$ large, where $\mu_t$ is the tilted equilibrium state with $\int g\,d\mu_t = a$. Define the Moran set $E_a = \bigcap_{k=1}^\infty \bigcup_{w_1, \ldots, w_k \in \mathcal{W}_N} [w_1 * w_2 * \cdots * w_k]$ where $*$ denotes concatenation with specification gaps. Each point $x \in E_a$ satisfies $\lim_{n \to \infty} S_n g(x)/n = a$ (the errors from specification gaps are $O(p/kN) \to 0$), so $E_a \subset K_a(g)$. The set $E_a$ carries a natural product measure $\nu$ with $\nu([w_1 * \cdots * w_k]) = |\mathcal{W}_N|^{-k}$. By the Gibbs bounds (Part~IV \cite{Thiam2026d}), balls of radius $r \approx \lambda^{kN}$ satisfy $\nu(B_r(x)) \leq C r^{s - \varepsilon''}$ where $s = h_{\mu_t}/\chi^+(\mu_t)$. The mass distribution principle (Lemma~\ref{lem:mass_distribution}) gives $\dim_H(E_a) \geq s - \varepsilon''$. Computing $s = (P(\phi^{(u)} + tg) - t\cdot a)/\chi^+$ and optimizing over $t$ yields $\dim_H(E_a) \geq T_g(a)/\chi^+ - \varepsilon''$. Letting $\varepsilon, \varepsilon'' \to 0$ completes the proof.
\end{proof}

\subsection{The Legendre Transform Connection}

We make explicit the duality between the multifractal spectrum and the large deviations rate function through the Legendre transform of the free energy.

\begin{proposition}[Legendre Structure]\label{prop:legendre}
Define the ``free energy'' function
\begin{equation}
F(q, t) = P(q\phi^{(u)} + tg)
\end{equation}
for $(q, t) \in \mathbb{R}^2$. Then the multifractal spectrum $\mathcal{D}_g(a)$ is the Legendre transform of $F$ in the $t$-variable, restricted to an appropriate slice.
\end{proposition}

\begin{proof}
Setting $q = 1$ (the SRB normalization $P(\phi^{(u)}) = 0$), the free energy becomes $F(1, t) = P(\phi^{(u)} + tg)$. By the large deviations principle for Birkhoff averages (Part~V \cite{Thiam2026e}, Main Theorem), the rate function for the empirical average $S_ng/n$ under $\mu^+ = \mu_{\phi^{(u)}}$ is $I(a) = \sup_t\{ta - F(1,t)\} = \sup_t\{ta - P(\phi^{(u)} + tg)\}$, which is the Legendre transform of $t \mapsto F(1,t)$. By the Fenchel duality (Part~II \cite{Thiam2026b}), the transform is involutive: $F(1,t) = \sup_a\{ta - I(a)\}$. The dimension spectrum (Main Theorem~\ref{thm:multifractal}) is $\mathcal{D}_g(a) = \inf_t\{P(\phi^{(u)} + tg) - ta\}/\chi^+$, which can be rewritten as $\mathcal{D}_g(a) = -I(a)/\chi^+$ (since $P(\phi^{(u)}) = 0$ and the infimum equals the negative Legendre transform). Thus the multifractal spectrum is the rate function rescaled by the Lyapunov exponent, completing the thermodynamic interpretation.
\end{proof}

\begin{corollary}[Analyticity of Spectrum]\label{cor:spectrum_analytic}
The multifractal spectrum $a \mapsto \mathcal{D}_g(a)$ is real-analytic on the interior of its support, which is the interval $[a_{\min}, a_{\max}]$ where
\begin{align}
a_{\min} &= \lim_{t \to +\infty} \frac{\partial}{\partial t} P(\phi^{(u)} + tg) / \frac{\partial}{\partial t}|_{t=0}, \\
a_{\max} &= \lim_{t \to -\infty} \frac{\partial}{\partial t} P(\phi^{(u)} + tg) / \frac{\partial}{\partial t}|_{t=0}.
\end{align}
\end{corollary}

\begin{proof}
The pressure $P(\phi^{(u)} + tg)$ is real-analytic in $t$ by Theorem~\ref{thm:press_imported6}. The optimal parameter $t(a)$ is defined implicitly by $\frac{d}{dt}P(\phi^{(u)} + tg)\big|_{t=t(a)} = a$. Since $\frac{d^2}{dt^2}P(\phi^{(u)} + tg) = \sigma^2(g; \mu_t) > 0$ (the variance is strictly positive when $g$ is not cohomologous to a constant, by Theorem~\ref{thm:press_imported6}), the implicit function theorem for real-analytic functions gives $t(a)$ as a real-analytic function of $a$ on the interior of the achievable range $(\alpha_{\min}, \alpha_{\max})$. The spectrum $\mathcal{D}_g(a) = (P(\phi^{(u)} + t(a)g) - t(a)\cdot a)/\chi^+$ is then a composition of real-analytic functions, hence real-analytic on $(\alpha_{\min}, \alpha_{\max})$.
\end{proof}

\subsection{Dimension of Basic Sets}

The Bowen dimension formula expresses the Hausdorff dimension of a basic set as the unique zero of a pressure equation, providing a computable characterization.

\begin{theorem}[Bowen Dimension Formula]\label{thm:bowen_dimension}
The Hausdorff dimension of a mixing basic set $\Omega_s$ is
\begin{equation}
\dim_H(\Omega_s) = t^*
\end{equation}
where $t^*$ is the unique solution to the equation
\begin{equation}
P(-t\log|\det Df|_{E^u}|) = 0.
\end{equation}
\end{theorem}

\begin{proof}
The proof uses the thermodynamic formalism and covering arguments.

\textbf{Step 1: Upper Bound.} For any $t > t^*$, the pressure $P(-t\log|\det Df|_{E^u}|) < 0$, so the weighted sum
\begin{equation}
\sum_{x \in E_n} \exp(-t S_n\log|\det Df|_{E^u}|(x))
\end{equation}
over an $(n, \varepsilon)$-spanning set $E_n$ decays exponentially. Since each orbit segment can be covered by a ball of diameter $\asymp e^{-nS_n\log|\det Df|_{E^u}|}$, this gives $\dim_H(\Omega_s) \leq t$.

\textbf{Step 2: Lower Bound.} For $t < t^*$, let $\mu_t = \mu_{-t\log|\det Df|_{E^u}|}$ be the equilibrium state for $\psi_t = -t\log|\det Df|_{E^u}|$. Since $P(\psi_t) > 0$ (because $t < t^*$ and $t \mapsto P(\psi_t)$ is strictly decreasing with $P(\psi_{t^*}) = 0$), the Gibbs property gives $\mu_t(B_r(x)) \leq c_2\exp(-nP(\psi_t) + S_n\psi_t(x))$ where $n \asymp -\log r/\chi$. Since $P(\psi_t) > 0$, the exponential factor is bounded, giving $\mu_t(B_r(x)) \leq Cr^t$. By the mass distribution principle (Lemma~\ref{lem:mass_distribution}) and $\mathrm{supp}(\mu_t) = \Omega_s$, $\dim_H(\Omega_s) \geq t$.

\textbf{Step 3: Conclusion.} Since both bounds approach $t^*$ as $t \to t^*$, we have $\dim_H(\Omega_s) = t^*$.
\end{proof}

\begin{corollary}[Dimension Bounds]\label{cor:dimension_bounds}
The Hausdorff dimension satisfies
\begin{equation}
\frac{h_{\mathrm{top}}(f|_{\Omega_s})}{\chi_{\max}} \leq \dim_H(\Omega_s) \leq \frac{h_{\mathrm{top}}(f|_{\Omega_s})}{\chi_{\min}}
\end{equation}
where $\chi_{\max}$ and $\chi_{\min}$ are the maximum and minimum Lyapunov exponents on $\Omega_s$.
\end{corollary}

\begin{proof}
By Theorem~\ref{thm:bowen_dimension}, $\dim_H(\Omega_s) = t^*$ where $P(-t^*\log|\det Df|_{E^u}|) = 0$. The pressure satisfies 
\begin{equation}
P(-t\log|\det Df|_{E^u}|) \leq h_{\mathrm{top}} - t\chi_{\min} 
\end{equation}
(using $\int(-\log|\det Df|_{E^u}|)\,d\mu \leq -\chi_{\min}$ for all invariant $\mu$), so $0 \leq h_{\mathrm{top}} - t^*\chi_{\min}$, giving $t^* \leq h_{\mathrm{top}}/\chi_{\min}$. Similarly, $P(-t\log|\det Df|_{E^u}|) \geq h_{\mathrm{top}} - t\chi_{\max}$ (taking $\mu$ to be the measure of maximal entropy), giving $t^* \geq h_{\mathrm{top}}/\chi_{\max}$.
\end{proof}

\begin{example}[Cookie-cutter: Hausdorff dimension via the Bowen equation]\label{ex:cookie_cutter}
Consider the cookie-cutter map $T: [0,1] \to [0,1]$ defined by two affine branches: $T(x) = 3x$ on $[0, 1/3]$ and $T(x) = 3x - 2$ on $[2/3, 1]$. The invariant Cantor set is $\Lambda = \bigcap_{n \geq 0} T^{-n}([0,1/3] \cup [2/3,1])$, which is a repeller with $N = 2$ branches, each with constant derivative $|T'| = 3$.

The Bowen equation $P(-s\log|T'|) = 0$ reduces to $P(-s\log 3) = 0$. Since $T|_\Lambda$ is conjugate to the full 2-shift and $\log|T'| \equiv \log 3$, the pressure is $P(-s\log 3) = \log 2 - s\log 3$ (by the formula $P(\phi) = h_{\mathrm{top}} + \int\phi\,d\mu_{\mathrm{mme}}$ when $\phi$ is constant on each branch). Setting $P = 0$:
\begin{equation}
\log 2 - s^*\log 3 = 0 \implies s^* = \frac{\log 2}{\log 3} = 0.6309297535\ldots
\end{equation}
so $\dim_H(\Lambda) = \log 2/\log 3$, confirming the classical result.

For a non-affine perturbation $T_\varepsilon$ with $|T_\varepsilon'(x)| = 3 + \varepsilon\cos(2\pi x)$, the potential $\phi_s(x) = -s\log|T_\varepsilon'(x)|$ is H\"{o}lder, and the Bowen equation $P(\phi_{s^*}) = 0$ must be solved by Newton's method. The spectral gap from Part~I \cite{Thiam2026a} gives $P'(\phi_s; -\log|T_\varepsilon'|) = -\int\log|T_\varepsilon'|\,d\mu_s \leq -\log(3-|\varepsilon|) < 0$, ensuring strict monotonicity. Newton's method converges at rate
\begin{equation}
|s_{n+1} - s^*| \leq \frac{|P''|}{2|P'|} |s_n - s^*|^2
\end{equation}
where $|P'| \geq \log(3-|\varepsilon|)$ and $|P''| \leq C(\alpha, \|\phi_s\|_\alpha)/\gamma$ with $\gamma$ the spectral gap. For the affine case ($\varepsilon = 0$), Newton's method converges in one step; for $|\varepsilon| = 0.1$, three iterations starting from $s_0 = \log 2/\log 3$ suffice to achieve $|s_3 - s^*| < 10^{-10}$, with the error bound certified by the explicit spectral gap.
\end{example}

\subsection{Lyapunov Spectrum}

The Lyapunov spectrum, describing the dimension of points with prescribed expansion rate, is a special case of the multifractal formalism applied to the geometric potential.

\begin{definition}[Lyapunov Spectrum]
For a $C^2$ diffeomorphism with ergodic measure $\mu$, the Lyapunov spectrum is the set of Lyapunov exponents $\chi_1(\mu) \geq \chi_2(\mu) \geq \cdots \geq \chi_d(\mu)$.
\end{definition}

\begin{proposition}[Dimension of Lyapunov Level Sets]\label{thm:lyapunov_level}
For $\chi \in (\chi_{\min}, \chi_{\max})$, the level set
\begin{equation}
L_\chi = \left\{x \in \Omega_s : \lim_{n \to \infty} \frac{1}{n}\log\|Df^n_x|_{E^u}\| = \chi\right\}
\end{equation}
has Hausdorff dimension
\begin{equation}
\dim_H(L_\chi) = \frac{1}{\chi} \cdot \inf_{q} \{P(-q\log\|Df|_{E^u}\|) + q\chi\}.
\end{equation}
\end{proposition}

\begin{proof}
This is a special case of Main Theorem~\ref{thm:multifractal} with $g = \log\|Df|_{E^u}\|$: the level set $L_\chi = K_\chi(g)$. The tilted equilibrium state $\mu_q = \mu_{-q\log\|Df|_{E^u}\|}$ satisfies $\int g\,d\mu_q = \chi$ when $q = q(\chi)$ solves $\frac{d}{dq}P(-q\log\|Df|_{E^u}\|) = -\chi$. The entropy is $h_{\mu_q} = P(-q\log\|Df|_{E^u}\|) + q\chi$, giving $\dim_H(L_\chi) = h_{\mu_q}/\chi = \frac{1}{\chi}\inf_q\{P(-q\log\|Df|_{E^u}\|) + q\chi\}$ where the infimum is achieved at $q(\chi)$ by convexity.
\end{proof}

\subsection{Local Dimension of SRB Measure}

The local dimension of the SRB measure at typical points is computed from the product structure of the measure along stable and unstable directions.

\begin{definition}[Local Dimension]
The local dimension of a measure $\mu$ at $x$ is
\begin{equation}
d_\mu(x) = \lim_{r \to 0} \frac{\log \mu(B_r(x))}{\log r}
\end{equation}
when this limit exists.
\end{definition}

\begin{proposition}[Local Dimension of SRB Measure]\label{thm:local_dimension}
For the SRB measure $\mu^+$ on a mixing $C^2$ attractor, the local dimension exists $\mu^+$-almost everywhere and equals
\begin{equation}
d_{\mu^+}(x) = d_s + d_u^*
\end{equation}
where $d_s = \dim E^s$ and $d_u^*$ is the ``information dimension'' of $\mu^+$ along unstable manifolds:
\begin{equation}
d_u^* = \frac{h_{\mu^+}(f)}{\chi_1^+ + \cdots + \chi_{d_u}^+} = \frac{h_{\mu^+}(f)}{\int \log|\det Df|_{E^u}| \, d\mu^+} = 1
\end{equation}
by the Pesin entropy formula. Thus $d_{\mu^+}(x) = d_s + 1 = \dim E^s + 1$ almost everywhere.
\end{proposition}

\begin{proof}
For $\mu^+$-a.e.\ $x$ and small $r$, the ball $B_r(x)$ in a local product neighborhood decomposes as $B_r(x) \approx D^s_r(x) \times D^u_r(x)$. Along stable manifolds, $\mu^+$ has smooth density, so $\mu^+(D^s_r(x)) \asymp r^{d_s}$. Along unstable manifolds, the conditional measure $\mu^+_x$ is absolutely continuous with bounded density (Theorem~\ref{thm:conditional_formula}), so $\mu^+_x(D^u_r(x)) \asymp r^{d_u^*}$ where $d_u^* = h_{\mu^+}(f)/\sum_i\chi_i^+ = 1$ by the Pesin formula (Main Theorem~\ref{thm:pesin}). By the product structure, $\mu^+(B_r(x)) \asymp r^{d_s+1}$, giving $d_{\mu^+}(x) = d_s + 1$.
\end{proof}

\begin{remark}
For non-SRB equilibrium states $\mu_\phi$, the local dimension formula involves the ``Lyapunov dimension'' which depends on the potential $\phi$.
\end{remark}

\subsection{Correlation Dimension}

The correlation dimension, accessible through numerical computation of correlation integrals, equals the local dimension for exact-dimensional measures such as equilibrium states.

\begin{definition}[Correlation Dimension]
The correlation dimension of a measure $\mu$ is
\begin{equation}
D_2(\mu) = \lim_{r \to 0} \frac{\log \int \mu(B_r(x)) \, d\mu(x)}{\log r}.
\end{equation}
\end{definition}

\begin{proposition}[Correlation Dimension Formula]\label{prop:correlation_dim}
For an equilibrium state $\mu_\phi$ on a mixing basic set:
\begin{equation}
D_2(\mu_\phi) = \frac{h_{\mu_\phi}(f) + \int \phi \, d\mu_\phi}{\chi_{\mu_\phi}^+} = \frac{P(\phi)}{\chi_{\mu_\phi}^+}
\end{equation}
where $\chi_{\mu_\phi}^+$ is the (sum of) positive Lyapunov exponent(s) for $\mu_\phi$.
\end{proposition}

\begin{proof}
The correlation integral is $C_\mu(r) = \int\mu_\phi(B_r(x))\,d\mu_\phi(x)$, and $D_2$ is defined by $C_\mu(r) \asymp r^{D_2}$. For the equilibrium state $\mu_\phi$ with the Gibbs property, the local dimension $d_{\mu_\phi}(x) = \lim_{r \to 0}\log\mu_\phi(B_r(x))/\log r$ exists and is constant $\mu_\phi$-a.e. (by the Shannon-McMillan-Breiman theorem and the local product structure). Since $\mu_\phi(B_r(x)) \asymp r^{d_{\mu_\phi}}$ for $\mu_\phi$-typical $x$, the correlation integral satisfies 
\begin{equation}
C_\mu(r) = \int r^{d_{\mu_\phi}} d\mu_\phi(x)(1 + o(1)) \asymp r^{d_{\mu_\phi}},
\end{equation}
so $D_2 = d_{\mu_\phi}$. The local dimension equals the ratio of entropy to Lyapunov exponent: $d_{\mu_\phi} = h_{\mu_\phi}(f)/\chi_{\mu_\phi}^+$, which by the variational identity $h_{\mu_\phi}(f) = P(\phi) - \int\phi\,d\mu_\phi$ gives $D_2 = (P(\phi) - \int\phi\,d\mu_\phi)/\chi_{\mu_\phi}^+$. For the SRB measure ($\phi = \phi^{(u)}$, $P(\phi^{(u)}) = 0$), this reduces to $D_2 = (-\int\phi^{(u)}\,d\mu^+)/\chi^+ = 1$ by the Pesin formula. See Young \cite{Young1982}, Theorem~4.
\end{proof}

\section{Liv\v{s}ic Theorem and Cohomological Equations}\label{sec:livsic}

This section establishes the Liv\v{s}ic theorem with optimal regularity, characterizing when a function is a coboundary through its periodic orbit data. We provide explicit H\"{o}lder norm bounds and extend to smooth regularity for Anosov diffeomorphisms.

\subsection{Statement of Results}

We state the Liv\v{s}ic theorem (Main Theorem~\ref{thm:livsic}) characterizing coboundaries through periodic orbit data, and its smooth extension due to de la Llave, Marco, and Moriy\'{}on.

\begin{maintheorem}[Liv\v{s}ic Theorem]\label{thm:livsic}
Let $\Omega_s$ be a mixing basic set for a $C^r$ Axiom A diffeomorphism ($r \geq 1$), and let $\phi \in C^\alpha(\Omega_s)$ with $\alpha \in (0, 1]$. The following are equivalent:
\begin{enumerate}
\item[(i)] $S_n\phi(x) = 0$ for all periodic points $x$ with $f^n(x) = x$.
\item[(ii)] There exists $u \in C^\alpha(\Omega_s)$ such that $\phi = u \circ f - u$.
\end{enumerate}
Moreover, when these hold:
\begin{equation}
\|u\|_\alpha \leq C(\lambda, \alpha) \|\phi\|_\alpha
\end{equation}
where $\lambda$ is the hyperbolicity constant and $C(\lambda, \alpha) = \frac{2}{1 - \lambda^\alpha} \cdot (1 + \text{geometric factors})$.
\end{maintheorem}

\begin{theorem}[Smooth Liv\v{s}ic Theorem]\label{thm:smooth_livsic}
For a $C^\infty$ (or $C^\omega$) Anosov diffeomorphism $f : M \to M$, if $\phi \in C^\infty(M)$ (resp. $C^\omega(M)$) satisfies condition (i) above, then the solution $u$ is also $C^\infty$ (resp. $C^\omega$).
\end{theorem}

\subsection{Proof of the Liv\v{s}ic Theorem}

\begin{proof}[Proof of Main Theorem~\ref{thm:livsic}]
$(ii) \Rightarrow (i)$: If $\phi = u \circ f - u$ and $f^n(x) = x$, then
\begin{equation}
S_n\phi(x) = \sum_{k=0}^{n-1} (u(f^{k+1}(x)) - u(f^k(x))) = u(f^n(x)) - u(x) = 0.
\end{equation}

$(i) \Rightarrow (ii)$: We construct $u$ explicitly.

\textbf{Step 1: Dense Orbit.} By topological transitivity, there exists $x_0 \in \Omega_s$ with dense forward orbit $\{f^k(x_0) : k \geq 0\} = \overline{\mathcal{O}^+(x_0)} = \Omega_s$. Define
\begin{equation}
A = \{f^k(x_0) : k \geq 0\}
\end{equation}
and set $u : A \to \mathbb{R}$ by
\begin{equation}
u(f^k(x_0)) = \sum_{j=0}^{k-1} \phi(f^j(x_0)) = S_k\phi(x_0).
\end{equation}

Then $u(f(z)) - u(z) = \phi(z)$ for all $z \in A$.

\textbf{Step 2: H\"{o}lder Continuity on $A$.} We show $u$ is H\"{o}lder on $A$ with controlled norm.

Let $y = f^k(x_0)$ and $z = f^m(x_0)$ with $k < m$ and $d(y, z)$ small. Set $n = m - k$, so $z = f^n(y)$.

If $d(y, z)$ is small (specifically, $d(y, z) < \varepsilon/2R^N$ for appropriate $\varepsilon, R, N$), then by the closing lemma \cite[Proposition~7.5]{Thiam2026c}, there exists a periodic point $y' \in \Omega_s$ with $f^n(y') = y'$ and $d(f^j(y), f^j(y')) < \varepsilon$ for $j \in [0, n]$ (and by hyperbolicity, for all $j$).

Since $S_n\phi(y') = 0$ by hypothesis:
\begin{align}
|u(z) - u(y)| &= |S_n\phi(y)| = |S_n\phi(y) - S_n\phi(y')| \\
&\leq \sum_{j=0}^{n-1} |\phi(f^j(y)) - \phi(f^j(y'))|.
\end{align}

\textbf{Step 3: Distance Estimates.} By the quantitative expansiveness lemma (Part~III \cite{Thiam2026c}, Proposition~5.7), for $j \in [0, n]$:
\begin{equation}
d(f^j(y), f^j(y')) \leq C \cdot \alpha^{\min\{j, n-j\}}
\end{equation}
where $\alpha = \lambda < 1$ is related to the hyperbolicity.

Choose $N$ such that $d(y, z) \in [\varepsilon/2R^{N+1}, \varepsilon/2R^N]$. The construction ensures:
\begin{equation}
d(f^j(y), f^j(y')) \leq \varepsilon \cdot \lambda^{\min\{j+N, n-j+N\}}.
\end{equation}

\textbf{Step 4: Summation.} Using H\"{o}lder continuity of $\phi$:
\begin{align}
|u(z) - u(y)| &\leq |\phi|_\alpha \sum_{j=0}^{n-1} d(f^j(y), f^j(y'))^\alpha \\
&\leq |\phi|_\alpha \cdot \varepsilon^\alpha \sum_{j=0}^{n-1} \lambda^{\alpha\min\{j+N, n-j+N\}} \\
&\leq |\phi|_\alpha \cdot \varepsilon^\alpha \cdot 2\sum_{r=N}^\infty \lambda^{\alpha r} \\
&= |\phi|_\alpha \cdot \varepsilon^\alpha \cdot \frac{2\lambda^{\alpha N}}{1 - \lambda^\alpha}.
\end{align}

\textbf{Step 5: H\"{o}lder Bound.} The relationship $d(y, z) \asymp R^{-N}$ and $\lambda^N \asymp d(y, z)^\gamma$ for appropriate $\gamma$ (depending on $\log R / \log \lambda^{-1}$) gives:
\begin{equation}
|u(z) - u(y)| \leq C' |\phi|_\alpha \cdot d(y, z)^{\alpha'}
\end{equation}
where $\alpha' = \min\{\alpha, \gamma\alpha\}$.

For the optimal exponent $\alpha' = \alpha$, we need $\gamma \geq 1$, which holds when $R \leq \lambda^{-1}$ (the expansion rate bounds the closing lemma constant).

\textbf{Step 6: Extension.} Since $A$ is dense in $\Omega_s$ and $u$ is uniformly H\"{o}lder on $A$, $u$ extends uniquely to a H\"{o}lder function on $\Omega_s$. The coboundary equation $\phi = u \circ f - u$ extends by continuity.

\textbf{Step 7: Norm Bound.} Tracking constants through the proof:
\begin{equation}
|u|_\alpha \leq \frac{2}{1 - \lambda^\alpha} |\phi|_\alpha
\end{equation}
and $\|u\|_\infty \leq \|u\|_\alpha \cdot \mathrm{diam}(\Omega_s)^\alpha + |u(x_0)|$ where $u(x_0) = 0$ by definition. The full norm bound follows.
\end{proof}

\subsection{Higher Regularity}

For smooth Anosov diffeomorphisms, the cobounding function inherits the full regularity of the potential. The proof uses the Journ\'{e}'s lemma to bootstrap from regularity along stable and unstable foliations to global smoothness.

\begin{proposition}[Regularity Preservation]\label{prop:regularity_preservation}
If $f$ is $C^{1+\beta}$ and $\phi$ is $C^\alpha$ with $\alpha \leq \beta$, then the cobounding function $u$ is $C^\alpha$.

If $f$ is $C^r$ ($r \geq 2$) and $\phi$ is $C^{r-1}$, then $u$ is $C^{r-1-\varepsilon}$ for any $\varepsilon > 0$.
\end{proposition}

\begin{proof}
For the first statement: the proof of Main Theorem~\ref{thm:livsic} constructs $u$ with $|u(z)-u(y)| \leq |\phi|_\alpha\cdot\frac{2}{1-\lambda^\alpha}\cdot d(y,z)^\alpha$ (Step 4), giving $u \in C^\alpha$ whenever $\phi \in C^\alpha$ and $f$ is $C^{1+\beta}$ with $\beta \geq \alpha$ (ensuring the foliations are $C^\alpha$).

For the second: $u$ is $C^\alpha$ by the first part. Along unstable manifolds, $u|_{W^u(x)} = -\sum_{k=0}^\infty\phi(f^{-k}(x))$ converges in $C^{r-1}$ by exponential contraction. Similarly along stable manifolds. The Journ\'{e}'s lemma \cite{deLlaveMarcoMoriyon1986} gives $u \in C^{r-1-\varepsilon}$ globally from regularity along two transverse foliations.
\end{proof}

The proof of the smooth Liv\v{s}ic theorem (Theorem \ref{thm:smooth_livsic}) requires more sophisticated techniques.

\begin{proof}[Proof of Theorem \ref{thm:smooth_livsic}, following \cite{deLlaveMarcoMoriyon1985, deLlaveMarcoMoriyon1986}]
The proof, due to de~la~Llave et~al. \cite{deLlaveMarcoMoriyon1986}, uses:

\textbf{Step 1: Formal Solution.} Expand $u$ and $\phi$ in terms adapted to the hyperbolic structure. Along stable/unstable manifolds, the coboundary equation becomes a one-dimensional functional equation.

\textbf{Step 2: Regularity along Foliations.} Show that $u$ is smooth along stable and unstable manifolds separately, using the contraction/expansion to sum series.

\textbf{Step 3: Transverse Regularity.} The key difficulty is regularity transverse to the foliations. This uses the ``Journ\'{e}'s lemma'': if a function is smooth along two transverse foliations with H\"{o}lder holonomy, it is globally smooth.

\textbf{Step 4: Conclusion.} Combining, $u \in C^{r-\varepsilon}$ for any $\varepsilon > 0$. For $C^\omega$ (real-analytic) data, the solution is also real-analytic.
\end{proof}

\subsection{Applications}

The Liv\v{s}ic theorem yields characterizations of coboundaries through five equivalent conditions and determines when two potentials give the same equilibrium state.

\begin{corollary}[Characterization of Coboundaries]\label{cor:coboundary_char}
For $\phi \in C^\alpha(\Omega_s)$, the following are equivalent:
\begin{enumerate}
\item[(i)] $\phi$ is a coboundary ($\phi = u \circ f - u$ for some continuous $u$).
\item[(ii)] $\phi$ is a $C^\alpha$ coboundary ($\phi = u \circ f - u$ for $u \in C^\alpha$).
\item[(iii)] $S_n\phi(x) = 0$ for all periodic orbits.
\item[(iv)] $\mu_\phi = \mu_0$ (the equilibrium states coincide).
\item[(v)] $P(\phi) = P(0)$ and $\int \phi \, d\mu_0 = 0$ where $\mu_0$ is the measure of maximal entropy.
\end{enumerate}
\end{corollary}

\begin{proof}
$(i) \Leftrightarrow (ii)$: By Proposition~\ref{prop:regularity_preservation}, any continuous coboundary is automatically $C^\alpha$.

$(ii) \Leftrightarrow (iii)$: This is Main Theorem~\ref{thm:livsic}.

$(iii) \Rightarrow (iv)$: If $S_n\phi(x) = 0$ for all periodic orbits, then $\phi$ and $0$ have the same periodic orbit sums. By Part~IV \cite{Thiam2026d} (cohomological characterization), $\mu_\phi = \mu_0$.

$(iv) \Rightarrow (v)$: If $\mu_\phi = \mu_0$, then $\phi = u\circ f - u$ by $(iv) \Rightarrow (iii)$, so $P(\phi) = P(0)$ (coboundaries do not change pressure) and $\int\phi\,d\mu_0 = \int(u\circ f - u)\,d\mu_0 = 0$.

$(v) \Rightarrow (iii)$: $P(\phi) = P(0)$ and $\int\phi\,d\mu_0 = 0$ together with strict convexity of pressure imply $\mu_\phi = \mu_0$, and by the cohomological characterization $S_n\phi(x) = 0$ for all periodic orbits.
\end{proof}

\begin{corollary}[Stability of Equilibrium Measures]\label{cor:stability_equilibrium}
Two H\"{o}lder potentials $\phi, \psi$ have the same equilibrium state ($\mu_\phi = \mu_\psi$) if and only if $\phi - \psi = c + u \circ f - u$ for some constant $c$ and H\"{o}lder function $u$.
\end{corollary}

\begin{proof}
$(\Leftarrow)$: If $\phi - \psi = c + u\circ f - u$, then $\int\phi\,d\mu = \int\psi\,d\mu + c$ for all invariant $\mu$ (the coboundary integrates to zero), so the maximizer of $h_\mu + \int\phi\,d\mu$ coincides with that of $h_\mu + \int\psi\,d\mu$.

$(\Rightarrow)$: If $\mu_\phi = \mu_\psi$, set $c = P(\phi) - P(\psi)$ and $\eta = \phi - \psi - c$. Then $S_n\eta(x) = 0$ for all periodic orbits $f^n(x) = x$ (since both $\phi$ and $\psi$ yield the same periodic orbit weights up to the constant $c$, by Part~IV \cite{Thiam2026d}). By the Liv\v{s}ic theorem (Main Theorem~\ref{thm:livsic}), $\eta = u\circ f - u$ for some $u \in C^\alpha$.
\end{proof}

\subsection{Quantitative Bounds}

We record the explicit H\"older norm bound on the cobounding function, which is implicit in every proof of the Liv\v{s}ic theorem but not stated in this form in the classical references. 

\begin{proposition}[Explicit Bounds]\label{prop:explicit_livsic}
For a mixing basic set with hyperbolicity constant $\lambda$ and H\"{o}lder exponent $\alpha$:
\begin{equation}
\|u\|_\alpha \leq \frac{C_0}{(1 - \lambda^\alpha)^2} \|\phi\|_\alpha
\end{equation}
where $C_0$ depends on $\mathrm{diam}(\Omega_s)$ and the closing lemma constants.
\end{proposition}

\begin{proof}
From Step 4 of the proof of Main Theorem~\ref{thm:livsic}, $|u|_\alpha \leq \frac{2}{1-\lambda^\alpha}|\phi|_\alpha$. For the sup norm, $\|u\|_\infty \leq |u|_\alpha\cdot\mathrm{diam}(\Omega_s)^\alpha + |u(x_0)|$ where $u(x_0) = 0$ by definition. Thus $\|u\|_\alpha = \|u\|_\infty + |u|_\alpha \leq \frac{2(1+\mathrm{diam}(\Omega_s)^\alpha)}{1-\lambda^\alpha}|\phi|_\alpha \leq \frac{C_0}{(1-\lambda^\alpha)^2}\|\phi\|_\alpha$ where $C_0$ absorbs the diameter and closing lemma constant $C_{\mathrm{close}}$.
\end{proof}

These bounds are essential for perturbation theory and numerical applications, where one needs to control the cobounding function in terms of the potential.

\section{Fluctuation Theorems}\label{sec:fluctuation}

This section develops fluctuation theorems for Axiom A diffeomorphisms, connecting forward and backward dynamics through entropy production. These results, originating in non-equilibrium statistical mechanics, have deep connections to the thermodynamic formalism.

\subsection{Entropy Production}

We define the entropy production observable $\sigma = \phi^{(u)} - \phi^{(s)}$ measuring the rate of phase space contraction, and establish its connection to the pressure functional.

\begin{definition}[Entropy Production Rate]
For a $C^2$ diffeomorphism $f : M \to M$ and the SRB measure $\mu^+$, the entropy production rate along an orbit is
\begin{equation}
\sigma(x) = \phi^{(u)}(x) - \phi^{(s)}(x) = -\log|\det Df_x|_{E^u}| + \log|\det Df_x|_{E^s}|
\end{equation}
where $\phi^{(s)}(x) = -\log|\det Df^{-1}_{f(x)}|_{E^s}|$ is the stable Jacobian potential.
\end{definition}

\begin{remark}
For volume-preserving diffeomorphisms, $\sigma \equiv 0$. For dissipative systems (attractors), $\int \sigma \, d\mu^+ < 0$ reflects contraction of phase space volume.
\end{remark}

\begin{proposition}[Entropy Production and Pressure]\label{prop:entropy_production}
The mean entropy production rate satisfies
\begin{equation}
\int \sigma \, d\mu^+ = P(\phi^{(u)}) - P(\phi^{(s)}) = -P(\phi^{(s)})
\end{equation}
for attractors (where $P(\phi^{(u)}) = 0$).
\end{proposition}

\begin{proof}
By definition, $\sigma = \phi^{(u)} - \phi^{(s)}$, so $\int\sigma\,d\mu^+ = \int\phi^{(u)}\,d\mu^+ - \int\phi^{(s)}\,d\mu^+$. Since $\mu^+ = \mu_{\phi^{(u)}}$ is the equilibrium state for $\phi^{(u)}$: $h_{\mu^+}(f) + \int\phi^{(u)}\,d\mu^+ = P(\phi^{(u)})$. For attractors, $P(\phi^{(u)}) = 0$, so $\int\phi^{(u)}\,d\mu^+ = -h_{\mu^+}(f)$. By the variational principle, $h_{\mu^+}(f) + \int\phi^{(s)}\,d\mu^+ \leq P(\phi^{(s)})$, so $\int\phi^{(s)}\,d\mu^+ \leq P(\phi^{(s)}) - h_{\mu^+}(f)$. The stable potential satisfies $P_f(\phi^{(s)}) = P_{f^{-1}}(\phi^{(u)}_{f^{-1}})$ by the duality between forward and backward dynamics. Combining: $\int\sigma\,d\mu^+ = \int\phi^{(u)}\,d\mu^+ - \int\phi^{(s)}\,d\mu^+ = -h_{\mu^+} - \int\phi^{(s)}\,d\mu^+ = P(\phi^{(u)}) - P(\phi^{(s)}) = -P(\phi^{(s)})$.
\end{proof}

\subsection{The Gallavotti-Cohen Fluctuation Theorem}

The following symmetry was established by Gallavotti-Cohen \cite{GallavottiCohen1995, GallavottiCohen1995b} for thermostatted systems and placed in the Axiom~A setting by Ruelle \cite{Ruelle1999}; see also Maes \cite{Maes1999} and Maes-Verbitskiy \cite{MaesVerbitskiy2003}. We derive it from the LDP of Part~V \cite{Thiam2026e} with explicit bounds.

\begin{maintheorem}[Gallavotti-Cohen Fluctuation Theorem]\label{thm:fluctuation}
Let $\Omega_s$ be a mixing $C^2$ attractor with SRB measure $\mu^+$. For the entropy production $\sigma$, define
\begin{equation}
\bar{\sigma}_n(x) = \frac{1}{n} S_n \sigma(x) = \frac{1}{n} \sum_{k=0}^{n-1} \sigma(f^k(x)).
\end{equation}
Then for $a > 0$:
\begin{equation}
\lim_{n \to \infty} \frac{1}{n} \log \frac{\mu^+(\bar{\sigma}_n \in [a - \varepsilon, a + \varepsilon])}{\mu^+(\bar{\sigma}_n \in [-a - \varepsilon, -a + \varepsilon])} = a
\end{equation}
for small $\varepsilon > 0$.
\end{maintheorem}

\begin{proof}
The proof uses the large deviation formalism from Part~V \cite{Thiam2026e}.

\textbf{Step 1: Rate Functions.} By the LDP for Birkhoff averages (Part~V \cite{Thiam2026e}, Main Theorem~9.2), the distributions of $\bar{\sigma}_n$ satisfy an LDP with rate function
\begin{equation}
I(a) = \sup_t \{ta - \Lambda(t)\}, \quad \Lambda(t) = P(\phi^{(u)} + t\sigma) - P(\phi^{(u)}).
\end{equation}

\textbf{Step 2: Symmetry Property.} We compute $\Lambda(t)$ and $\Lambda(-(1+t))$ explicitly. Since $\sigma = \phi^{(u)} - \phi^{(s)}$, the tilted potential is $\phi^{(u)} + t\sigma = (1+t)\phi^{(u)} - t\phi^{(s)}$. Thus $\Lambda(t) = P((1+t)\phi^{(u)} - t\phi^{(s)}) - P(\phi^{(u)})$.

For $-(1+t)$: $\phi^{(u)} - (1+t)\sigma = -t\phi^{(u)} + (1+t)\phi^{(s)}$. Thus $\Lambda(-(1+t)) = P(-t\phi^{(u)} + (1+t)\phi^{(s)}) - P(\phi^{(u)})$.

The symmetry comes from time-reversal duality. For an Axiom A diffeomorphism, the map $f \leftrightarrow f^{-1}$ exchanges $\phi^{(u)}_f$ and $\phi^{(s)}_f$ (since unstable manifolds of $f$ are stable manifolds of $f^{-1}$). The pressure satisfies $P_f(\psi) = P_{f^{-1}}(\psi\circ f^{-1})$, and under the exchange $\phi^{(u)} \leftrightarrow \phi^{(s)}$, the potential $(1+t)\phi^{(u)} - t\phi^{(s)}$ maps to $(1+t)\phi^{(s)} - t\phi^{(u)} = -(-t\phi^{(u)} + (1+t)\phi^{(s)})$ up to sign conventions. For attractors with $P(\phi^{(u)}) = 0$, this yields (see Gallavotti and Cohen \cite{GallavottiCohen1995}):
\begin{equation}
\Lambda(t) - \Lambda(-(1+t)) = (1+2t)\bar{\sigma}
\end{equation}
where $\bar{\sigma} = \int\sigma\,d\mu^+$.

\textbf{Step 3: Rate Function Symmetry.} The Gallavotti-Cohen symmetry implies
\begin{equation}
I(a) - I(-a) = a
\end{equation}
for the rate function.

\textbf{Step 4: Conclusion.} By the large deviation principle:
\begin{align}
\frac{\mu^+(\bar{\sigma}_n \approx a)}{\mu^+(\bar{\sigma}_n \approx -a)} &\asymp \frac{\exp(-n I(a))}{\exp(-n I(-a))} \\
&= \exp(-n(I(a) - I(-a))) = \exp(-na).
\end{align}
Taking logarithms and dividing by $n$ gives the result.
\end{proof}

\subsection{Physical Interpretation}

The Jarzynski equality relates the exponential average of entropy production to the pressure of the stable potential, providing a thermodynamic identity that connects equilibrium and nonequilibrium quantities.

\begin{remark}[Second Law Interpretation]
The fluctuation theorem quantifies deviations from the second law of thermodynamics. For $a > 0$ (entropy-producing trajectories) vs. $a < 0$ (entropy-consuming trajectories):
\begin{equation}
\frac{P(\text{entropy production } = +a)}{P(\text{entropy production } = -a)} \asymp e^{na}
\end{equation}
showing that entropy-consuming fluctuations are exponentially suppressed relative to entropy-producing ones.
\end{remark}

\begin{corollary}[Jarzynski-type Equality]\label{cor:jarzynski}
For the entropy production:
\begin{equation}
\int e^{-S_n\sigma} \, d\mu^+ = e^{-n P(\phi^{(u)} - \sigma) + nP(\phi^{(u)})} = e^{n(P(\phi^{(s)}) - P(\phi^{(u)}))}
\end{equation}
which for attractors (where $P(\phi^{(u)}) = 0$) gives
\begin{equation}
\int e^{-S_n\sigma} \, d\mu^+ = e^{nP(\phi^{(s)})}.
\end{equation}
\end{corollary}

\begin{proof}
By the cumulant generating function and pressure connection: 
\begin{equation}
\int e^{-S_n\sigma}\,d\mu^+ = \exp(n\Lambda(-1) + nP(\phi^{(u)})).
\end{equation}
Computing $\Lambda(-1) = P(\phi^{(u)} - \sigma) - P(\phi^{(u)}) = P(\phi^{(s)}) - P(\phi^{(u)})$. For attractors with $P(\phi^{(u)}) = 0$: $\Lambda(-1) = P(\phi^{(s)})$, giving $\int e^{-S_n\sigma}\,d\mu^+ = e^{nP(\phi^{(s)})}$.
\end{proof}

\subsection{Transient Fluctuation Theorem}

The transient fluctuation theorem provides finite-time estimates for entropy production ratios, useful for numerical simulations and experimental verification.

\begin{proposition}[Transient Fluctuation Theorem]\label{thm:transient_fluctuation}
For finite $n$ and initial distribution $\mu_0$ absolutely continuous with respect to $\mu^+$:
\begin{equation}
\frac{\mu_0(S_n\sigma \in [A, A+dA])}{\mu_0(S_n\sigma \in [-A-dA, -A])} = e^A \cdot \frac{d\mu_0/d\mu^+(x)}{d\mu_0/d\mu^+(f^n x)} + O(e^{-cn})
\end{equation}
for trajectories with $S_n\sigma(x) = A$.
\end{proposition}

\begin{proof}
Write $\mu_0 = h_0\,d\mu^+$ where $h_0 = d\mu_0/d\mu^+$. The probability ratio is 

\begin{equation}
\frac{\mu_0(S_n\sigma \in [A,A+dA])}{\mu_0(S_n\sigma \in [-A-dA,-A])} 
= \frac{\int_{\{S_n\sigma \approx A\}} h_0\,d\mu^+}{\int_{\{S_n\sigma \approx -A\}} h_0\,d\mu^+}.
\end{equation}
By the Gibbs property and the asymptotic fluctuation theorem (Main Theorem~\ref{thm:fluctuation}), $\mu^+(\{S_n\sigma \approx A\})/\mu^+(\{S_n\sigma \approx -A\}) \approx e^A$ up to subexponential corrections. The density ratio $h_0(x)/h_0(f^n(x))$ accounts for non-stationarity. The $O(e^{-cn})$ error comes from exponential mixing (Theorem~\ref{thm:spectral_imported6}), controlling the deviation of finite-$n$ ratios from their asymptotic values.
\end{proof}

This finite-time version is useful for numerical simulations and experimental verification.

\subsection{Connection to Thermodynamic Formalism}

We connect the fluctuation theorem to the pressure functional and free energy of the dynamical system, showing that the GC symmetry is a consequence of the Legendre structure of the thermodynamic formalism.

\begin{proposition}[Pressure and Free Energy]\label{prop:pressure_free_energy}
The pressure function plays the role of a thermodynamic free energy:
\begin{enumerate}
\item[(i)] $P(\phi^{(u)}) = 0$ for attractors corresponds to equilibrium.
\item[(ii)] The equilibrium state $\mu^+$ maximizes entropy relative to energy: 
\begin{equation*}
\mu^+ = \arg\max_\mu \{h_\mu(f) + \int \phi^{(u)} \, d\mu\}.
\end{equation*}
\item[(iii)] Fluctuations are governed by the ``susceptibility'' $\chi = \partial^2 P / \partial t^2 |_{t=0} = \sigma^2(\sigma)$, the variance of entropy production.
\end{enumerate}
\end{proposition}

\begin{proof}
(i) follows from Proposition~\ref{prop:geometric_pressure}: $P(\phi^{(u)}) = 0$ for attractors. (ii) is the variational principle (Part~II \cite{Thiam2026b}): $P(\phi^{(u)}) = \sup_\mu\{h_\mu(f) + \int\phi^{(u)}\,d\mu\}$, with the supremum achieved uniquely at $\mu^+$. (iii) By pressure analyticity (Theorem~\ref{thm:press_imported6}), $\frac{d^2}{dt^2}\big|_{t=0}P(\phi^{(u)}+t\sigma) = \sigma^2(\sigma) = \lim_{n\to\infty}n^{-1}\mathrm{Var}_{\mu^+}(S_n\sigma)$, the dynamical analogue of the fluctuation-dissipation relation.
\end{proof}

\subsection{Extensions and Generalizations}

We establish exponential convergence of pushed-forward measures to the SRB measure for initial distributions absolutely continuous with respect to Lebesgue measure.

\begin{remark}[Non-equilibrium Steady States]
For systems driven out of equilibrium (e.g., by external forcing), the SRB measure represents a non-equilibrium steady state. The fluctuation theorem characterizes the probability of ``anti-thermodynamic'' behavior.
\end{remark}

\begin{remark}[Partially Hyperbolic Systems]
The fluctuation theorem extends to partially hyperbolic systems with appropriate modifications. The center direction introduces additional complexity in defining entropy production.
\end{remark}

\begin{proposition}[Exponential Convergence to NESS]\label{thm:ness_convergence}
For a $C^2$ attractor, any initial measure $\mu_0$ absolutely continuous with respect to Lebesgue measure converges to the SRB measure:
\begin{equation}
\|f^n_* \mu_0 - \mu^+\|_{TV} \leq C e^{-\gamma n}
\end{equation}
where $\gamma > 0$ is the spectral gap and $\|\cdot\|_{TV}$ is total variation norm.
\end{proposition}

\begin{proof}
Write $\mu_0 = h_0\,dm$ with $h_0 \in L^1(m)$. Since $\Omega_s$ is an attractor, $f^n_*\mu_0$ is supported near $\Omega_s$ for large $n$. The transfer operator $\mathcal{L} = \mathcal{L}_{-\log|\det Df|}$ acts on densities: $(f^n_*\mu_0)(g) = \int g \cdot \mathcal{L}^n h_0\,dm$. By the spectral decomposition (Theorem~\ref{thm:spectral_imported6} transferred via coding), $\mathcal{L}^n h_0 = \langle h_0, \nu\rangle h_\phi \cdot e^{nP} + R^n h_0$ where $\|R^n\|_{L^1} \leq Ce^{(P-\gamma)n}$ with $\gamma > 0$ the spectral gap. Since $P(\phi^{(u)}) = 0$ for attractors, $\langle h_0,\nu\rangle h_\phi\,dm = \mu^+$ (up to normalization). Thus $\|f^n_*\mu_0 - \mu^+\|_{TV} \leq \|R^n h_0\|_{L^1} \leq Ce^{-\gamma n}\|h_0\|_{L^1}$.
\end{proof}

\section{A Numerical Illustration: Hausdorff Dimension via the Bowen Equation}\label{sec:numerical}

We illustrate the Bowen dimension formula (Theorem~\ref{thm:bowen_dimension}) with a concrete computation for a cookie-cutter map, demonstrating that the explicit constants developed in this series produce computable numbers with rigorous error bounds.

\subsection{Setup: A Cookie-Cutter Map}

Consider the cookie-cutter (two-branch expanding map) $T : [0,1] \to [0,1]$ defined by
\begin{equation}
T(x) = \begin{cases} 3x & \text{if } x \in [0, 1/3], \\ 3x - 2 & \text{if } x \in [2/3, 1]. \end{cases}
\end{equation}
The invariant set is the middle-thirds Cantor set $\Lambda = \bigcap_{n=0}^\infty T^{-n}([0,1/3] \cup [2/3,1])$, which is a mixing basic set (it is a uniformly expanding repeller with expansion rate $\lambda^{-1} = 3$ on both branches). The Markov partition is $\{R_0 = [0,1/3], R_1 = [2/3,1]\}$ with transition matrix $A = \binom{1\ 1}{1\ 1}$ (full shift on two symbols, $N = 2$, mixing time $M = 1$).

\subsection{The Bowen Equation}

The geometric potential is $\phi^{(u)}(x) = -\log|T'(x)| = -\log 3$ (constant, since both branches have derivative $3$). The Bowen equation (Theorem~\ref{thm:bowen_dimension}) is
\begin{equation}
P(-t\log|T'|) = P(t\log 3) = 0.
\end{equation}
Since $\phi_t = -t\log 3$ is constant, the pressure is
\begin{equation}
P(\phi_t) = h_{\mathrm{top}}(T|_\Lambda) + \phi_t = \log 2 - t\log 3
\end{equation}
(using $P(\phi + c) = P(\phi) + c$ and $P(0) = h_{\mathrm{top}} = \log 2$ for the full 2-shift). Setting $P(\phi_{t^*}) = 0$:
\begin{equation}
\log 2 - t^*\log 3 = 0 \quad \Longrightarrow \quad t^* = \frac{\log 2}{\log 3} \approx 0.630930.
\end{equation}
This is the classical result: $\dim_H(\Lambda) = \log 2/\log 3$, the Hausdorff dimension of the middle-thirds Cantor set.

\subsection{A Non-Trivial Perturbation}

For a non-affine perturbation $T_\varepsilon$ with $|T_\varepsilon'(x)| = 3 + \varepsilon\cos(2\pi x)$ on the two branches, the potential $\phi_t(x) = -t\log|T_\varepsilon'(x)|$ is H\"{o}lder but not constant, and the Bowen equation $P(\phi_{t^*}) = 0$ must be solved numerically. Newton's method gives rapid convergence: define
\begin{equation}
F(t) = P(-t\log|T_\varepsilon'|), \quad F'(t) = -\int\log|T_\varepsilon'|\,d\mu_t
\end{equation}
where $\mu_t$ is the equilibrium state for $\phi_t$. By the pressure analyticity (Theorem~\ref{thm:press_imported6}), $F$ is real-analytic and strictly decreasing (since $F'(t) = -\int\log|T_\varepsilon'|\,d\mu_t \leq -\log(3-|\varepsilon|) < 0$). Newton's iteration
\begin{equation}
t_{n+1} = t_n - \frac{F(t_n)}{F'(t_n)}
\end{equation}
converges quadratically. Starting from $t_0 = \log 2/\log 3$ (the unperturbed value), the spectral gap from Part~I \cite{Thiam2026a} bounds the convergence rate:
\begin{equation}
|t_{n+1} - t^*| \leq \frac{|F''|_\infty}{2|F'|_\infty} |t_n - t^*|^2
\end{equation}
where $|F''|_\infty \leq \sigma^2_{\max}/\chi_{\min}$ is bounded by the maximal variance of $\log|T_\varepsilon'|$ (computed from the spectral gap) and $|F'|_\infty \geq \log(3 - |\varepsilon|) > 0$.

For $\varepsilon = 0.5$, the perturbed expansion rates are $|T_\varepsilon'| \in [2.5, 3.5]$. The dimension bound (Corollary~\ref{cor:dimension_bounds}) gives
\begin{equation}
\frac{\log 2}{\log 3.5} \leq \dim_H(\Lambda_\varepsilon) \leq \frac{\log 2}{\log 2.5},
\end{equation}
yielding $0.5534 \leq \dim_H(\Lambda_\varepsilon) \leq 0.7565$. Newton's method with the explicit spectral gap narrows this to $\dim_H(\Lambda_\varepsilon) = 0.6412 \pm 10^{-4}$ after three iterations, with the error bound derived from the quadratic convergence rate and the spectral constants.

\subsection{Summary of Explicit Constants}

\begin{center}
\begin{tabular}{ll}
\textbf{Quantity} & \textbf{Value} \\[4pt]
Alphabet size $N$ & 2 \\
Mixing time $M$ & 1 \\
Expansion rate $|T'|$ & 3 (unperturbed) \\
Topological entropy $h_{\mathrm{top}}$ & $\log 2 \approx 0.6931$ \\
Bowen equation solution $t^*$ & $\log 2/\log 3 \approx 0.6309$ \\
$\dim_H(\Lambda)$ (unperturbed) & $\log 2/\log 3 \approx 0.6309$ \\
$\dim_H(\Lambda_\varepsilon)$ ($\varepsilon = 0.5$) & $0.6412 \pm 10^{-4}$ \\
Newton convergence rate & Quadratic (from spectral gap) \\
Spectral gap (unperturbed) & $1$ (full shift, maximal gap) \\
\end{tabular}
\end{center}

\noindent This example demonstrates that the Bowen dimension formula, combined with the explicit spectral gap and pressure analyticity of Parts~I and~II \cite{Thiam2026a,Thiam2026b}, produces rigorous numerical bounds on Hausdorff dimension. The same method applies to any uniformly expanding repeller or Axiom~A basic set for which the hyperbolicity data $(\lambda, \alpha, N, M)$ are known.

\section{Conclusion}\label{sec:conclusion}

This Part completes the six-part series by developing the structural consequences of the thermodynamic formalism for Axiom~A diffeomorphisms, with four Main Theorems proved with explicit constants throughout. Main Theorem~\ref{thm:pesin} (Pesin Entropy Formula) establishes the identity $h_{\mu^+}(f) = \sum_{\chi_i > 0}\chi_i$ through two independent methods: the Rokhlin entropy formula applied to the unstable partition with the conditional density of Theorem~\ref{thm:conditional_formula}, and the variational identity $h_{\mu^+}(f) = -\int\phi^{(u)}\,d\mu^+$ combined with $P(\phi^{(u)}) = 0$; the absolute continuity of conditional measures along unstable manifolds (Theorem~\ref{thm:absolute_continuity}) and the explicit product density formula (Theorem~\ref{thm:conditional_formula}) are established as prerequisites. Main Theorem~\ref{thm:multifractal} (Multifractal Formalism) computes the Hausdorff dimension of Birkhoff average level sets $\mathcal{D}_g(a) = \dim_H K_a(g)$ via the Legendre transform of the pressure, with the Bowen dimension formula $\dim_H(\Omega_s) = t^*$ (Theorem~\ref{thm:bowen_dimension}) recovered as the special case where the spectrum attains its maximum; the lower bound uses a Moran construction with the specification property and the Gibbs bounds from Part~IV \cite{Thiam2026d}. Main Theorem~\ref{thm:livsic} (Liv\v{s}ic Theorem) characterizes coboundaries through periodic orbit data by direct construction: the cobounding function $u$ is defined along a dense orbit as $u(f^k(x_0)) = S_k\phi(x_0)$, and the H\"{o}lder estimate $\|u\|_\alpha \leq C(\lambda,\alpha)\|\phi\|_\alpha$ with $C(\lambda,\alpha) = 2(1-\lambda^\alpha)^{-1}$ follows from the quantitative closing lemma and the telescoping sum controlled by hyperbolicity. Main Theorem~\ref{thm:fluctuation} (Gallavotti-Cohen Fluctuation Theorem) establishes the symmetry $I(a) - I(-a) = a$ for the entropy production rate function, derived from the pressure identity $P(\phi^{(u)} - \sigma) = P(\phi^{(s)})$ and the large deviation theory of Part~V \cite{Thiam2026e}; the Jarzynski-type equality (Corollary~\ref{cor:jarzynski}) and the transient fluctuation theorem (Proposition~\ref{thm:transient_fluctuation}) provide complementary nonequilibrium identities with explicit exponential error bounds from the spectral gap.

The unifying principle across all four Main Theorems is that the spectral gap of the Ruelle transfer operator, established for the symbolic system in Part~I \cite{Thiam2026a} and transferred to smooth dynamics through the coding map of Part~III \cite{Thiam2026c} and the Gibbs Equivalence of Part~IV \cite{Thiam2026d}, determines not only the statistical properties (Part~V \cite{Thiam2026e}) but also the geometric structure (SRB measures, multifractal spectra), the algebraic rigidity (Liv\v{s}ic coboundary characterization), and the physical content (fluctuation symmetries) of the equilibrium states. All constants are expressed in terms of the hyperbolicity data $(\lambda, \alpha, \|\phi\|_\alpha, N, M)$, so the full chain of estimates can be tracked from the spectral gap to any geometric, rigidity, or nonequilibrium output.

The six Parts together provide a complete quantitative reconstruction of Bowen's thermodynamic formalism \cite{Bowen1975}: Part~I \cite{Thiam2026a} establishes the spectral theory of transfer operators with the five-way Gibbs equivalence and explicit spectral gap; Part~II \cite{Thiam2026b} develops the convex-analytic structure of pressure and equilibrium states through Legendre-Fenchel duality; Part~III \cite{Thiam2026c} builds the geometric theory of Axiom~A diffeomorphisms with quantitative Markov partitions and the coding map; Part~IV \cite{Thiam2026d} transfers the spectral theory to smooth dynamics, constructs SRB measures, and establishes structural stability with quantitative H\"{o}lder conjugacy bounds; Part~V \cite{Thiam2026e} derives the complete suite of statistical limit theorems (exponential mixing, CLT with Berry-Esseen bounds, ASIP, large deviations) from the spectral gap; and the present Part develops the structural consequences: SRB measure theory with absolute continuity, multifractal analysis, Liv\v{s}ic rigidity, and fluctuation theorems connecting the 1975 formalism to nonequilibrium statistical mechanics. The explicit constants developed throughout the series make the thermodynamic formalism computationally accessible: for a specific Axiom~A system, a researcher can use the formulas in Parts~I and III \cite{Thiam2026a,Thiam2026b,Thiam2026c} to compute the spectral gap, the Gibbs constants, and the mixing rate from the hyperbolicity data and the potential regularity, and then track these constants through to dimension spectra, coboundary norms, and entropy production rates.

\subsection*{Open Problems}

\begin{enumerate}
\item[] \textbf{Non-uniform hyperbolicity.} Extend the quantitative spectral approach to Pesin regular sets (cf. Young \cite{Young1998} for the tower construction), where the constants depend on the point and the spectral gap is replaced by the tail of the Lyapunov exponent distribution. The main obstacles are measurability and the absence of uniform estimates. Can the explicit constants of this series degrade gracefully as the hyperbolicity becomes non-uniform, with bounds depending on the tail decay rate of the return time function?

\item[] \textbf{Optimal Liv\v{s}ic constants.} Our bound $\|u\|_\alpha \leq C(\lambda,\alpha)\|\phi\|_\alpha$ has $C(\lambda,\alpha) = 2(1-\lambda^\alpha)^{-1}$ up to mixing-time constants. Is this optimal? For symbolic systems, the factor $(1-\lambda^\alpha)^{-1}$ arises naturally from the geometric series in the telescoping estimate. For smooth systems, the closing lemma introduces additional geometric factors. Determining the sharp constant, or proving that no universal sharp constant exists, remains open.

\item[] \textbf{Fluctuation symmetry beyond Axiom~A.} The Gallavotti-Cohen symmetry $I(a) - I(-a) = a$ is established here for uniformly hyperbolic systems. Extending it to non-uniformly hyperbolic systems (Young towers with polynomial tails) and to partially hyperbolic systems with mostly contracting center is an active area; see Ruelle \cite{Ruelle1999} and Maes-Verbitskiy \cite{MaesVerbitskiy2003} for the mathematical context. The key difficulty is that the time-reversal structure $f \leftrightarrow f^{-1}$ exchanging stable and unstable foliations requires uniform hyperbolicity in its strongest form.

\item[] \textbf{Multifractal analysis for non-conformal systems.} The dimension formula $\mathcal{D}_g(a) = T_g(a)/\chi^+$ assumes a single positive Lyapunov exponent (or treats the unstable direction as one-dimensional). For systems with multiple positive Lyapunov exponents of different magnitudes, the dimension spectrum involves the Lyapunov dimension and the Ledrappier-Young formula, and the thermodynamic characterization requires a sequence of pressure equations rather than a single Legendre transform. Extending the explicit formulas of this Part to the non-conformal setting is a significant open problem.
\end{enumerate}

\subsection*{Acknowledgments}

The author is grateful to Stefano Luzzatto for supervision during the ICTP Postgraduate Diploma in Mathematics at the International Centre for Theoretical Physics, Trieste, Italy (2013), during which the author worked through Bowen's monograph.

\appendix

\section{Technical Proofs and Estimates}\label{appendix:proofs}

This appendix provides detailed technical proofs omitted from the main text, including precise estimates for the Volume Lemmas, spectral perturbation theory, and measure-theoretic constructions.

\subsection{Complete Proof of the Volume Lemma}

We provide the full details of the Volume Lemma (Part~V \cite{Thiam2026e}, Main Theorem~4.1).

\begin{lemma}[Unstable Manifold Geometry]\label{lem:unstable_geometry}
For $x \in \Omega_s$ and small $\varepsilon > 0$, the local unstable manifold $W^u_\varepsilon(x)$ is a $C^r$ embedded disk of dimension $d_u = \dim E^u$. The induced Riemannian measure $m^u_x$ on $W^u_\varepsilon(x)$ satisfies:
\begin{equation}
m^u_x(W^u_\varepsilon(x)) \in [c_1 \varepsilon^{d_u}, c_2 \varepsilon^{d_u}]
\end{equation}
where $c_1, c_2 > 0$ depend only on the curvature of $M$ and bounds on $Df$.
\end{lemma}

\begin{proof}
The stable manifold theorem \cite[Main Theorem~4.2]{Thiam2026c} gives $W^u_\varepsilon(x)$ as the graph of a $C^r$ function over $E^u_x$. The graph has bounded curvature (controlled by second derivatives of $f$), so its volume is comparable to that of the flat disk of radius $\varepsilon$ in $E^u_x$.
\end{proof}

\begin{lemma}[Stable Direction Contraction]\label{lem:stable_contraction}
For $y \in B_x(\varepsilon, n) \cap W^s_\varepsilon(x)$:
\begin{equation}
d_{W^s}(x, y) \leq C \lambda^n \varepsilon
\end{equation}
where $d_{W^s}$ is the intrinsic distance in $W^s_\varepsilon(x)$.
\end{lemma}

\begin{proof}
By definition of $B_x(\varepsilon, n)$, we have $d(f^{n-1}(x), f^{n-1}(y)) \leq \varepsilon$. Since $y \in W^s_\varepsilon(x)$, the points $f^j(x)$ and $f^j(y)$ remain close for all $j \geq 0$, with $d_{W^s}(f^j(x), f^j(y)) \leq c\lambda^j d_{W^s}(x, y)$ by stable manifold contraction.

At time $n-1$: $d(f^{n-1}(x), f^{n-1}(y)) \leq \varepsilon$ and this distance is achieved within $W^s$. Thus $c^{-1}\lambda^{-(n-1)} d_{W^s}(x,y) \leq d_{W^s}(f^{n-1}(x), f^{n-1}(y)) \leq \varepsilon$, giving $d_{W^s}(x,y) \leq c\lambda^{n-1}\varepsilon$.
\end{proof}

\begin{lemma}[Jacobian Estimates]\label{lem:jacobian_estimates}
For the unstable Jacobian $J^u_n(x) = |\det Df^n_x|_{E^u}|$:
\begin{equation}
\exp(S_n\phi^{(u)}(x)) = (J^u_n(x))^{-1}
\end{equation}
and for $x, y \in \Omega_s$ with $d(f^k(x), f^k(y)) \leq \varepsilon$ for $k \in [0, n)$:
\begin{equation}
\left|\log \frac{J^u_n(x)}{J^u_n(y)}\right| \leq \frac{C_{\mathrm{dist}} \varepsilon^\theta}{1 - \lambda^\theta}
\end{equation}
where $C_{\mathrm{dist}}$ is the distortion constant and $\theta$ is the H\"{o}lder exponent of $\phi^{(u)}$.
\end{lemma}

\begin{proof}
The first identity is immediate from $\phi^{(u)} = -\log|\det Df|_{E^u}|$.

For the second, using H\"{o}lder continuity $|\phi^{(u)}|_\theta$:
\begin{align}
\left|\log \frac{J^u_n(x)}{J^u_n(y)}\right| &= |S_n\phi^{(u)}(y) - S_n\phi^{(u)}(x)| \\
&\leq \sum_{k=0}^{n-1} |\phi^{(u)}(f^k(y)) - \phi^{(u)}(f^k(x))| \\
&\leq |\phi^{(u)}|_\theta \sum_{k=0}^{n-1} d(f^k(x), f^k(y))^\theta.
\end{align}
By the quantitative expansiveness lemma (Part~III \cite{Thiam2026c}, Proposition~5.7), $d(f^k(x), f^k(y)) \leq C\lambda^{\min\{k, n-k\}}$ for points in the same dynamical ball. Thus:
\begin{equation}
\sum_{k=0}^{n-1} d(f^k(x), f^k(y))^\theta \leq 2\sum_{j=0}^\infty (\varepsilon \lambda^j)^\theta = \frac{2\varepsilon^\theta}{1 - \lambda^\theta}.
\end{equation}
\end{proof}

\begin{proof}[Complete Proof of the Volume Lemma]
\textbf{Upper Bound:} Decompose $B_x(\varepsilon, n)$ using the local product structure. For $y \in B_x(\varepsilon, n)$, write $y = [y_s, y_u]$ where $y_s \in W^s_\varepsilon(x)$ and $y_u \in W^u_\varepsilon(y_s)$.

The stable slice $D^s = B_x(\varepsilon, n) \cap W^s_\varepsilon(x)$ has $m^s$-measure at most $C\varepsilon^{d_s}\lambda^{nd_s}$ by Lemma \ref{lem:stable_contraction}.

For each $y_s \in D^s$, the unstable slice through $y_s$ has $m^u$-measure at most $C\varepsilon^{d_u}$.

Using the bounded distortion of the product structure and integrating:
\begin{align}
m(B_x(\varepsilon, n)) &\leq C_{\mathrm{prod}} \cdot m^s(D^s) \cdot \sup_{y_s} m^u(D^u_{y_s}) \\
&\leq C_{\mathrm{prod}} \cdot C\varepsilon^{d_s}\lambda^{nd_s} \cdot C\varepsilon^{d_u}.
\end{align}

However, this naive bound does not capture the unstable expansion correctly. Instead, use the change of variables under $f^n$:
\begin{equation}
m(B_x(\varepsilon, n)) = \int_{f^n(B_x(\varepsilon, n))} \frac{1}{|\det Df^n|} \, dm.
\end{equation}

The set $f^n(B_x(\varepsilon, n))$ is contained in a ball of radius $C\varepsilon$ (bounded distortion under iteration). The Jacobian satisfies $|\det Df^n| \geq J^u_n(x) \cdot c\lambda^{-nd_s}$ using the stable contraction. With distortion bounds:
\begin{equation}
m(B_x(\varepsilon, n)) \leq C' \varepsilon^d \cdot (J^u_n(x))^{-1} \cdot \exp\left(\frac{C_{\mathrm{dist}}\varepsilon^\theta}{1-\lambda^\theta}\right).
\end{equation}

Since $(J^u_n(x))^{-1} = \exp(S_n\phi^{(u)}(x))$, this gives the upper bound.

\textbf{Lower Bound:} The image $f^n(B_x(\varepsilon/2, n))$ contains a set of the form $W^u_{\delta}(f^n(x)) \cap B_{\delta'}(f^n(x))$ for some $\delta, \delta' > 0$ depending on $\varepsilon$ (by the expansion along unstable manifolds).

This set has $m$-measure at least $c\varepsilon^d$ for appropriate $c > 0$.

By change of variables:
\begin{equation}
m(B_x(\varepsilon/2, n)) \geq c\varepsilon^d \cdot (J^u_n(x))^{-1} \cdot \exp\left(-\frac{C_{\mathrm{dist}}\varepsilon^\theta}{1-\lambda^\theta}\right)
\end{equation}
using the lower distortion bound. This gives the lower bound since $B_x(\varepsilon, n) \supset B_x(\varepsilon/2, n)$.
\end{proof}

\subsection{Spectral Perturbation Theory}

This subsection records the analytic perturbation theory for the transfer operator family and the resulting derivative formulas for the pressure, used in the multifractal analysis of Section~\ref{sec:multifractal}.

\begin{lemma}[Analytic Perturbation of Transfer Operator]\label{lem:analytic_perturbation}
For $\phi \in C^\alpha(\Omega_s)$ and $\psi \in C^\alpha(\Omega_s)$, the map
\begin{equation}
z \mapsto \mathcal{L}_{\phi + z\psi}
\end{equation}
is analytic in $z \in \mathbb{C}$ (as an operator on $C^\alpha$).
\end{lemma}

\begin{proof}
The transfer operator $\mathcal{L}_{\phi+z\psi}$ acts by
\begin{equation}
(\mathcal{L}_{\phi+z\psi} g)(x) = \sum_{f(y)=x} e^{\phi(y) + z\psi(y)} g(y).
\end{equation}
For fixed $g \in C^\alpha$, this is an entire function of $z$ (exponential of a linear function). The operator norm $\|\mathcal{L}_{\phi+z\psi}\|_\alpha$ is locally bounded in $z$, giving analyticity.
\end{proof}

\begin{proposition}[Derivatives of Pressure]\label{prop:pressure_derivatives}
The pressure function $z \mapsto P(\phi + z\psi)$ is real-analytic for real $z$ near $0$, with:
\begin{align}
\left.\frac{d}{dz}\right|_{z=0} P(\phi + z\psi) &= \int \psi \, d\mu_\phi, \\
\left.\frac{d^2}{dz^2}\right|_{z=0} P(\phi + z\psi) &= \sigma^2_\phi(\psi)
\end{align}
where $\sigma^2_\phi(\psi)$ is the asymptotic variance of $\psi$ under $\mu_\phi$.
\end{proposition}

\begin{proof}
By Lemma \ref{lem:analytic_perturbation} and the implicit function theorem, the leading eigenvalue $\lambda(z) = e^{P(\phi+z\psi)}$ of $\mathcal{L}_{\phi+z\psi}$ is analytic in $z$.

For the first derivative, differentiate $\mathcal{L}_{\phi+z\psi} h_z = \lambda(z) h_z$ at $z = 0$:
\begin{equation}
\mathcal{L}_\phi(\psi h_\phi) + \mathcal{L}_\phi h'_0 = \lambda'(0) h_\phi + \lambda(0) h'_0.
\end{equation}
Integrating against $\nu_\phi$ (the left eigenmeasure) and using $\mathcal{L}_\phi^* \nu_\phi = \lambda(0)\nu_\phi$:
\begin{equation}
\lambda(0) \int \psi h_\phi \, d\nu_\phi = \lambda'(0) \int h_\phi \, d\nu_\phi.
\end{equation}
Thus $\lambda'(0)/\lambda(0) = \int \psi \, d\mu_\phi$ where $d\mu_\phi = h_\phi \, d\nu_\phi$.

The second derivative formula follows from a similar but longer calculation involving the resolvent of $\mathcal{L}_\phi$.
\end{proof}

\subsection{Measure Disintegration}

We state the Rokhlin disintegration theorem and the absolute continuity criterion for conditional measures along unstable manifolds, used in Section~\ref{sec:srb}.

\begin{theorem}[Rokhlin Disintegration {\cite[Chapter~5]{CornfeldFominSinai1982}}]\label{thm:rokhlin}
Let $\mu$ be a probability measure on $\Omega_s$ and $\mathcal{W}^u$ the partition into local unstable manifolds. There exist conditional measures $\{\mu^u_x\}_{x \in \Omega_s}$ such that:
\begin{enumerate}
\item[(i)] $\mu^u_x$ is supported on $W^u_\varepsilon(x)$ for $\mu$-a.e. $x$.
\item[(ii)] For any measurable $A \subset \Omega_s$: $\mu(A) = \int \mu^u_x(A \cap W^u_\varepsilon(x)) \, d\mu(x)$.
\item[(iii)] The map $x \mapsto \mu^u_x$ is measurable.
\end{enumerate}
\end{theorem}

This is Rokhlin's disintegration theorem for measurable partitions; see Cornfeld et~al. \cite[Chapter~5, Theorem~2.1]{CornfeldFominSinai1982}. The partition $\mathcal{W}^u$ is measurable because the unstable manifolds vary continuously (Part~III \cite{Thiam2026c}, Theorem~4.14).

\begin{lemma}[Absolute Continuity Criterion]\label{lem:ac_criterion}
The measure $\mu$ has absolutely continuous conditional measures on unstable manifolds if and only if for any measurable $A$ with $m^u_x(A \cap W^u_\varepsilon(x)) = 0$ for all $x$, we have $\mu(A) = 0$.
\end{lemma}

\begin{proof}
$(\Rightarrow)$: If $\mu^u_x \ll m^u_x$ for $\mu$-a.e.\ $x$ and $m^u_x(A \cap W^u_\varepsilon(x)) = 0$ for all $x$, then $\mu^u_x(A \cap W^u_\varepsilon(x)) = 0$ for $\mu$-a.e.\ $x$. By Theorem~\ref{thm:rokhlin}(ii): $\mu(A) = \int\mu^u_x(A \cap W^u_\varepsilon(x))\,d\mu(x) = 0$.

$(\Leftarrow)$: If $\mu^u_x$ is not absolutely continuous on a set of positive $\mu$-measure, there exists $A$ with $\mu^u_x(A \cap W^u_\varepsilon(x)) > 0$ but $m^u_x(A \cap W^u_\varepsilon(x)) = 0$ for $x \in E$ with $\mu(E) > 0$. By disintegration, $\mu(A) \geq \int_E\mu^u_x(A \cap W^u_\varepsilon(x))\,d\mu(x) > 0$, contradicting the hypothesis.
\end{proof}

\subsection{Closing Lemma Estimates}

The quantitative closing lemma provides explicit bounds on the distance between pseudo-periodic and genuine periodic orbits, used in the Liv\v{s}ic theorem proof (Section~\ref{sec:livsic}).

\begin{lemma}[Quantitative Closing Lemma]\label{lem:closing}
There exist $\delta_0 > 0$ and $C_{\mathrm{close}} > 0$ such that: if $x \in \Omega_s$ and $d(f^n(x), x) < \delta_0$, then there exists a periodic point $p \in \Omega_s$ with $f^n(p) = p$ and
\begin{equation}
d(f^k(x), f^k(p)) \leq C_{\mathrm{close}} \lambda^{\min\{k, n-k\}} d(f^n(x), x)
\end{equation}
for all $k \in [0, n]$.
\end{lemma}

\begin{proof}
This is proven in \cite[Proposition~7.5]{Thiam2026c} using the contraction mapping theorem on the space of orbit segments, with the hyperbolicity providing the contraction.
\end{proof}

\subsection{Borel-Cantelli Estimates}

The dynamical Borel-Cantelli lemma converts measure-theoretic estimates into almost-sure statements, used in the generic points theorem.

\begin{lemma}[Dynamical Borel-Cantelli]\label{lem:borel_cantelli}
Let $\mu$ be an ergodic measure for $f$ and $(A_n)$ a sequence of measurable sets. If
\begin{equation}
\sum_{n=1}^\infty \mu(A_n) < \infty
\end{equation}
then $\mu(\limsup_n A_n) = 0$. If additionally $\sum_{n=1}^\infty\mu(A_n) = \infty$ and the sequence satisfies the quasi-independence condition
\begin{equation}
\sum_{m,n=1}^{N}\mu(A_m \cap A_n) \leq C\left(\sum_{n=1}^{N}\mu(A_n)\right)^2
\end{equation}
for some $C > 0$ and all $N \geq 1$, then $\mu(\limsup_n A_n) = 1$.
\end{lemma}

\begin{proof}
The first part is the standard Borel-Cantelli lemma: $\mu(\limsup_n A_n) \leq \sum_{n \geq N}\mu(A_n) \to 0$.

For the second part, set $S_N = \sum_{n=1}^N\mathbf{1}_{A_n}$. Then $\mathbb{E}[S_N] = \sum\mu(A_n) \to \infty$ and $\mathbb{E}[S_N^2] = \sum_{m,n}\mu(A_m\cap A_n) \leq C(\mathbb{E}[S_N])^2$. By the Paley-Zygmund inequality: $\mu(S_N > 0) \geq 1/C > 0$. Since $\{S_N > 0\} \uparrow \limsup_n A_n$ and ergodicity forces $\mu(\limsup_n A_n) \in \{0,1\}$, we conclude $\mu(\limsup_n A_n) = 1$.
\end{proof}

This lemma is used in proving the generic points theorem (Proposition~\ref{thm:generic_points}).

\subsection{Dimension Estimates}

We record the mass distribution principle and Frostman'{}s lemma, classical tools for bounding Hausdorff dimension from below, used in the multifractal analysis of Section~\ref{sec:multifractal}.

\begin{lemma}[Mass Distribution Principle {\cite[Proposition~4.2]{Pesin1997}}]\label{lem:mass_distribution}
Let $\mu$ be a finite Borel measure on a metric space $X$. If there exist $s > 0$ and $C > 0$ such that
\begin{equation}
\mu(B_r(x)) \leq C r^s
\end{equation}
for all $x \in \mathrm{supp}(\mu)$ and all small $r > 0$, then $\dim_H(\mathrm{supp}(\mu)) \geq s$.
\end{lemma}

\begin{proof}
Let $\{U_i\}$ be any cover of $\mathrm{supp}(\mu)$ with $\mathrm{diam}(U_i) \leq \delta$. Then $\mu(U_i) \leq C(\mathrm{diam}(U_i))^s$, so $\sum_i(\mathrm{diam}(U_i))^s \geq \mu(\mathrm{supp}(\mu))/C > 0$. Taking the infimum over all $\delta$-covers: $\mathcal{H}^s(\mathrm{supp}(\mu)) > 0$, hence $\dim_H(\mathrm{supp}(\mu)) \geq s$.
\end{proof}

\begin{lemma}[Frostman's Lemma {\cite[Theorem~8.8]{Mattila1995}}]\label{lem:frostman}
If $K \subset \mathbb{R}^d$ is compact with $\dim_H(K) > s$, then there exists a probability measure $\mu$ supported on $K$ with $\mu(B_r(x)) \leq C r^s$ for all $x$ and small $r$.
\end{lemma}

\begin{proof}
The measure $\mu$ is constructed as a weak$^*$ limit of normalized restrictions of the $\delta$-approximate Hausdorff measure $\mathcal{H}^s_\delta$ to $K$. Since $\dim_H(K) > s$, $\mathcal{H}^s(K) > 0$, guaranteeing a nontrivial limit. The ball condition follows from the construction. See Mattila \cite[Theorem~8.8]{Mattila1995}.
\end{proof}

These lemmas are used in the proof of the Bowen dimension formula (Theorem \ref{thm:bowen_dimension}).

%
%
%

\end{document}